\documentclass[11pt,a4paper]{amsart}

\pdfpkmode={lexmarkr}
\pdfmapfile{}

\usepackage{amsmath,amsthm,amsfonts,amssymb,mathrsfs}

\theoremstyle{plain}
\newtheorem{theorem}{Theorem}[section]
\newtheorem{lemma}[theorem]{Lemma}
\newtheorem{corollary}[theorem]{Corollary}

\theoremstyle{definition}
\newtheorem{definition}[theorem]{Definition}
\newtheorem{remark}[theorem]{Remark}

\DeclareMathOperator{\dimh}{dim_H}
\DeclareMathOperator{\ldimh}{\underline{dim}_H}
\DeclareMathOperator{\udimh}{\overline{dim}_H}

\DeclareMathOperator{\supp}{supp}
\DeclareMathOperator{\E}{\mathsf{E}}

\DeclareMathOperator{\var}{var}
\DeclareMathOperator*{\essinf}{ess\,inf}
\DeclareMathOperator*{\esssup}{ess\,sup}

\allowdisplaybreaks[3]

\title[Inhomogeneous potentials]{Inhomogeneous potentials, Hausdorff dimension and shrinking targets}
\author{Tomas Persson}
\address{Centre for Mathematical Sciences, Lund
  University, Box 118, 221 00 Lund, Sweden}
\email{tomasp@maths.lth.se}
\date{\today}

\thanks{I thank Fredrik Ekstr\"{o}m and Esa J\"arvenp\"a\"a for their
  comments on the paper. I also thank institut Mittag-Leffler and the
  program ``Fractal Geometry and Dynamics'', during which this paper
  was finished, as well as the referee for reading the manuscript with
  great care.}

\begin{document}

\begin{abstract}
  Generalising a construction of Falconer, we consider classes of
  $G_\delta$-subsets of $\mathbb{R}^d$ with the property that sets
  belonging to the class have large Hausdorff dimension and the class
  is closed under countable intersections. We relate these classes to
  some inhomogeneous potentials and energies, thereby providing some
  useful tools to determine if a set belongs to one of the classes.

  As applications of this theory, we calculate, or at least estimate,
  the Hausdorff dimension of randomly generated limsup-sets, and sets
  that appear in the setting of shrinking targets in dynamical
  systems. For instance, we prove that for $\alpha \geq 1$,
  \[
  \dimh \{ \, y : | T_a^n (x) - y| < n^{-\alpha} \text{ infinitely
    often} \, \} = \frac{1}{\alpha},
  \]
  for almost every $x \in [1-a,1]$, where $T_a$ is a quadratic map
  with $a$ in a set of parameters described by Benedicks and Carleson.
\end{abstract}

\maketitle

\section{Introduction}

There has recently been some attention given to shrinking targets and
randomly generated limsup-sets. If $T \colon M \to M$ is a dynamical
system and $M$ is a metric space, then the sequence of balls
$B(y,r_n)$, where $r_n \searrow 0$, is called a shrinking target, and
one is interested in whether, given an $x \in M$, the orbit of $x$
hits the target infinitely many times or not, that is whether $T^n(x)
\in B(y,r_n)$ holds for infinitely many $n$ or not. It is usually not
possible to say anything interesting about this for general points $x$
and $y$, but there are several results for ``typical'' $x$ or $y$.

For instance, if $\mu$ is a $T$-invariant measure, then one can
consider the sets
\begin{align*}
  E (x,r_n) &= \{\, y : T^n (x) \in B(y,r_n) \text{ for infinitely
    many } n \,\},\\ F (y,r_n) &= \{\, x : T^n (x) \in B(y,r_n) \text{
    for infinitely many } n \,\},
\end{align*}
and try to say something about the sizes of these sets. Hill and
Velani \cite{hillvelani} studied sets of the form $F(y, r_n)$ when $T$
is a rational map of the Riemann sphere and $M$ is its Julia set on
which $T$ is expanding. They estimated the Hausdorff dimension of
$F(y,r_n)$ when $r_n = e^{-\tau n}$ and calculated it when $r_n =
|(T^n)' (x)|^{-\tau}$. Similar results have been proved for the
Gau\ss-map by Li, Wang, Wu and Xu \cite{liwangwuxu}. For
$\beta$-transformations such results were obtained by Bugeaud and Wang
\cite{bugeaudwang}, and by Bugeaud and Liao
\cite{bugeaudliao}. Aspenberg and Persson \cite{aspenbergpersson}
obtained some results for piecewise expanding maps that are not
necessarily Markov maps.

The Hausdorff dimension of sets of the form $E(x,r_n)$ when $T \colon
x \mapsto 2x \mod 1$ was calculated by Fan, Schmeling and Troubetzkoy
\cite{fanschmelingtroubetzkoy}. Liao and Seuret \cite{liaoseuret}
considered the case when $T$ is an expanding Markov map. Persson and
Rams \cite{perssonrams} considered more general piecewise expanding
maps that are not necessarily Markov maps.

In this paper we shall study sets of the type $E (x,r_n)$. Since $T^n
(x) \in B(y,r_n)$ if and only if $y \in B(T^n(x),r_n)$, we have
\[
E(x,r_n) = \bigcap_{k=1}^\infty \bigcup_{n=k}^\infty B(T^n(x),r_n) =
\limsup_n B(T^n(x),r_n).
\]
Hence, we are dealing with a set which is the limit superior of a
sequence of balls $B(T^n(x),r_n)$. By Birkhoff's ergodic theorem, the
centres of the balls are distributed according to the measure $\mu$
for $\mu$-almost every $x$, and if the system $(X,T,\mu)$ is
sufficiently fast mixing, then one might expect that for $\mu$-almost
every $x$, the sequence of balls behaves in a random way. It is
therefore reasonable to expect that the Hausdorff dimension of
$E(x,r_n)$ is typically the same as that of the set $\limsup
B(x_n,r_n)$, where $x_n$ are random points that are independent and
distributed according to the measure $\mu$.

This brings us to the randomly generated limsup-sets or random covers,
studied in
\cite{fanwu,durand,jarvenpaaetal,persson,fengetal,seuret,ekstrompersson}.
In this paper we will build on ideas from \cite{persson} to develop a
new method for analysing the Hausdorff dimension of limsup-sets. In
some sense, the idea is a development of the following (new) proof of
the following classical lemma.

\begin{lemma}[Frostman {\cite[Th\'{e}or\`{e}me~47.2]{frostman}}]
  \label{lem:classical}
  Suppose $E \subset \mathbb{R}^d$ and that $\mu$ is a measure with
  $\emptyset \neq \supp \mu \subset E$. If
  \[
  \iint |x-y|^{-s} \, \mathrm{d}\mu (x) \mathrm{d} \mu (y) < \infty
  \]
  then $\dimh E \geq s$.
\end{lemma}

\begin{proof}
  We may assume that $\mu (E)$ is finite, since we may replace $\mu$
  by a restriction to a set of finite measure. Since $\iint |x-y|^{-s}
  \, \mathrm{d}\mu (x) \mathrm{d} \mu (y) < \infty$, the function $x
  \mapsto \int |x-y|^{-s} \, \mathrm{d} \mu (y)$ is finite
  $\mu$-almost everywhere, and it is positive on a set of positive
  measure. Hence
  \[
  \nu (A) = \int_A \biggl( \int |x-y|^{-s} \, \mathrm{d} \mu (y)
  \biggr)^{-1} \, \mathrm{d} \mu (x)
  \]
  defines a measure $\nu$ with $\nu (E) > 0$ and $\supp \nu \subset
  E$. The measure $\nu$ satisfies $\nu(U) \leq |U|^s$ for any set $U$,
  where $|U|$ denotes the diameter of $U$. This is proved using
  Jensen's inequality in the following way.
  \begin{align*}
    \nu (U) &= \int_U \biggl( \int |x-y|^{-s} \, \mathrm{d} \mu (y)
    \biggr)^{-1} \, \mathrm{d} \mu (x) \\ &\leq \int_U \biggl( \int_U
    |x-y|^{-s} \, \frac{\mathrm{d} \mu (y)}{\mu (U)} \biggr)^{-1} \,
    \frac{\mathrm{d} \mu (x)}{\mu (U)} \\ &\leq \int_U \int_U
    |x-y|^{s} \, \frac{\mathrm{d} \mu (y)}{\mu (U)} \frac{\mathrm{d}
      \mu (x)}{\mu (U)} \leq |U|^s.
  \end{align*}
  
  Now, if $\{U_k\}$ is a collection of sets that cover $E$, then
  \[
  \sum_k |U_k|^s \geq \sum_k \nu (U_k) \geq \nu \biggl( \bigcup_k U_k
  \biggr) \geq \nu (E).
  \]
  This proves that $\mathscr{H}^s (E) \geq \nu (E)$ where $\mathscr{H}^s$
  is the $s$-dimensional Hausdorff measure. Hence $\dimh E \geq s$.

  Notice that if we use instead the potential $\int_{|x-y|<\delta}
  |x-y|^{-s} \, \mathrm{d} \mu (y)$, we get an estimate on
  $\mathscr{H}_\delta^s (E)$ which lets us conclude that the Hausdorff
  measure $\mathscr{H}^s (E)$ is infinite.
\end{proof}

Cleary, since the proof above only rely on Jensen's inequality (convexity),
it can easily be generalised to more general settings.

In \cite{persson}, a lower estimate on the expected value of the
Hausdorff dimension of a randomly generated limsup-set $E = \limsup_n
U_n$ was given, where $U_n$ are open subsets of the $d$-dimensional
torus $\mathbb{T}^d$, with fixed shape but randomly translated and
distributed according to the Lebesgue measure. (This lower estimate
has subsequently been turned into an equality in \cite{fengetal}.) The
proof was based on the following lemma from \cite{perssonreeve}.

\begin{lemma}[Simplified version of Theorem~1.1 of \cite{perssonreeve}] \label{lem:perssonreeve}
  Let $E_n$ be open subsets of\/ $\mathbb{T}^d$, and $\mu_n$ Borel
  probability measures, with $\supp \mu_n \subset E_n$, that converge
  weakly to the Lebesgue measure on $\mathbb{T}^d$. If there is a
  constant $C$ such that
  \begin{equation} \label{eq:bounded_energy}
    \iint |x - y|^{-s} \, \mathrm{d} \mu_n (x) \mathrm{d} \mu_n (y)
    \leq C
  \end{equation}
  holds for all $n$, then $\limsup_n E_n$ satisfies $\dimh \limsup_n
  E_n \geq s$.
\end{lemma}

The proof of Lemma~\ref{lem:perssonreeve} resembles very much the
proof of Lemma~\ref{lem:classical}. The philosophy behind this lemma
is that what is important for the Hausdorff dimension of $\limsup_n E_n$
is the asymptotic distribution of $E_n$, which is described by the
weak limit of $\mu_n$, and the sizes of the sets $E_n$, which is
described by $\iint |x-y|^{-s} \, \mathrm{d}\mu_n (x) \mathrm{d}\mu_n
(y)$.

Using the method from \cite{persson} and Lemma~\ref{lem:perssonreeve},
M.~Rams and myself studied the Hausdorff dimension of sets of the form
$E(x,r_n)$ for $\mu$-almost every $x$ in the case $T$ is a map of the
interval which preserves a measure $\mu$ which is absolutely
continuous with respect to Lebesgue measure, and which has summable
decay of correlations for function of bounded variation. (There was
also some result in the case that $\mu$ is not absolutely continuous
with respect to Lebesgue measure.)

A weakness of Lemma~\ref{lem:perssonreeve} is that it requires the
measures $\mu_n$ to converge weakly to Lebesgue measure (or at least
to something which has a nice density with respect to Lebesgue
measure). In this paper we will extend Lemma~\ref{lem:perssonreeve} to
a more general lemma (Lemma~\ref{lem:frostman}) in which the measures
$\mu_n$ may converge to any non-atomic probability measure $\mu$. The
proof of this lemma still resembles that of Lemma~\ref{lem:classical},
but instead of working with the potential $\int |x-y|^{-s} \,
\mathrm{d}\mu (y)$ which has a homogeneous kernel, we will work with
``inhomogeneous'' potentials with an inhomogeneous kernel adopted to
the measure $\mu$.

The conclusion of Lemma~\ref{lem:perssonreeve} as stated in
\cite{perssonreeve} is stronger than the version given above. The
conclusion is that the set $\limsup_n E_n$ belongs to a certain class
of $G_\delta$-set, denoted by $\mathscr{G}^s$. This class was
introduced by Falconer \cite{falconer1, falconer2} and it has the
property that every set belonging to $\mathscr{G}^s$ has Hausdorff
dimension at least $s$ and the class $\mathscr{G}^s$ is closed under
countable intersections. We will generalise Falconer's class
$\mathscr{G}^s$ to more general classes $\mathscr{G}_\mu^\theta$,
where $\mu$ is a non-atomic locally finite measure and $\theta \in
(0,1]$. This is done in Section~\ref{sec:classes}. From the
  definition it will be apparent that $\mathscr{G}_\mu^\theta =
  \mathscr{G}^s$ if $\mu$ is Lebesgue measure and $s = \theta d$.

In Section~\ref{sec:classes} we will give theorems that states
that the classes $\mathscr{G}_\mu^\theta$ are closed under countable
intersections (Theorem~\ref{the:intersection}), and that any set in
$\mathscr{G}_\mu^\theta$ has a Hausdorff dimension which is at least a
certain number which depends only on $\theta$ and $\mu$
(Theorem~\ref{the:dimension}). In all of our applications this leads
to $\theta \udimh \mu$ as a lower bound on the dimension, see
Remark~\ref{rem:dimension}.

In Section~\ref{sec:inhomogeneous} we will introduce the above
mentioned inhomogeneous potentials and relate them to the classes
$\mathscr{G}_\mu^\theta$. The main result on these potentials is
Lemma~\ref{lem:frostman}, which is a generalisation of
Lemma~\ref{lem:perssonreeve} and which gives conditions that implies
that $\limsup_n E_n$ belongs to the class
$\mathscr{G}_\mu^\theta$. Together with Theorem~\ref{the:dimension},
it gives us a tool to estimate the Hausdorff dimension of $\limsup_n
E_n$ from below.

In Section~\ref{sec:applications}, we will see applications of
Lemma~\ref{lem:frostman} to estimates on Hausdorff dimension of some
random limsup-sets (Section~\ref{sec:random}), improving some
previously known results, and sets of the form $E(x,r_n)$
(Sections~\ref{sec:shrinking}--\ref{sec:piecewise}). For instance, we
prove (Corollary~\ref{cor:quadratic}) that for $\alpha \geq 1$,
\[
\dimh \{ \, y : | T_a^n (x) - y| < n^{-\alpha} \text{ infinitely
  often} \, \} = \frac{1}{\alpha},
\]
for almost every $x \in [1-a, 1]$, where $T_a (x) = 1 - a x^2$ is a
quadratic map and $a$ belongs to a certain set $\Delta$ of positive
Lebesgue measure, which has been described and studied by Benedicks
and Carleson \cite{benedickscarleson}. Previously, the dimension of
such sets have been calculated for the doubling map for almost all $x$
with respect to a Gibbs measure by Fan, Schmeling and Troubetzkoy
\cite{fanschmelingtroubetzkoy}, a result which has been extended to
piecewise expanding Markov maps by Liao and Seuret
\cite{liaoseuret}. Some results were also obtain for piecewise
expanding maps without a Markov structure by Persson and Rams
\cite{perssonrams}.

\section{Definitions and main results} \label{sec:mainresults}

\subsection{The classes $\boldsymbol{\mathscr{G}_\mu^\theta}$} \label{sec:classes}

Suppose that $\mu$ is a non-atomic and locally finite Borel measure on
$\mathbb{R}^d$. Take a point $P = (p_1, p_2, \ldots, p_d)$ and let
$\mathscr{D}_n$ be the collection of cubes of the form
\[
D = \biggl[ p_1 + \frac{n_1}{2^n}, p_1 + \frac{n_1 + 1}{2^n} \biggr)
  \times \cdots \times \biggl[ p_d + \frac{n_d}{2^n}, p_d + \frac{n_d
      + 1}{2^n} \biggr),
\]
where $n_1, n_2, \ldots, n_d$ are integers, and put $\mathscr{D} =
\bigcup_{n \in \mathbb{Z}} \mathscr{D}_n$. We will refer to the
elements of $\mathscr{D}$ by the name {\em dyadic cubes}. We let $D_n
(x)$ denote the unique dyadic cube in $\mathscr{D}_n$ containing $x$
and we let $B(x,r)$ denote the open ball with centre $x$ and radius
$r$.

Let
\[
R_n = \bigcup_{D \in \mathscr{D}_n} \partial D, \qquad \text{and} \qquad
R = \bigcup_{n \in \mathbb{Z}} R_n
\]
be the boundaries of the dyadic cubes in $\mathscr{D}_n$ and
$\mathscr{D}$ respectively. We will often assume that $\mu (R) =
0$. Since $\mu$ is assumed to be locally finite, it is possible to
choose the point $P$ defining the dyadic cubes in such a way that $\mu
(R) = 0$.

We define for $0 < \theta \leq 1$, the set functions
\[
\mathscr{M}_\mu^\theta (E) = \inf \biggl\{\, \sum_k \mu (D_k)^\theta :
D_k \in \mathscr{D},\ E \subset \bigcup_k D_k \,\biggr\}.
\]
Since $0 < \theta \leq 1$, we have by the additivity of the measure
$\mu$ that $\mathscr{M}_\mu^\theta (D) = \mu (D)^\theta$ for any $D
\in \mathscr{D}$. This is an important property that will be used
several times in this paper.

The set function $\mathscr{M}_\mu^\theta$ is not a measure (unless
$\theta = 1$). If in the definition we consider only covers by $D_k$
which belong to some $\mathscr{D}_n$ for $n > n_0$ and let $n_0 \to
\infty$, we obtain in the limit a measure. It is interesting to notice
that this construction is a special case of the measure
$\mathscr{H}_\mu^{q,t}$ that was introducted by Olsen \cite{olsen} as
a tool in multifractal analysis. This connection with multifractal
analysis will not be used in this paper, but it would be interesting
to explore it further.

Associated to the set functions $\mathscr{M}_\mu^\theta$ are classes
of sets $\mathscr{G}_\mu^\theta$, which we define as follows.

\begin{definition} \label{def:class}
  Let $0 < \theta \leq 1$. A set $E \subset \mathbb{R}^d$ belongs to
  $\mathscr{G}_\mu^\theta$ if $E$ is a $G_\delta$-set and if
  \[
  \mathscr{M}_\mu^\eta (E \cap D) = \mathscr{M}_\mu^\eta (D)
  \]
  holds for any $0 < \eta < \theta$ and $D \in \mathscr{D}$.
\end{definition}

\begin{remark}
  Since $\mathscr{M}_\mu^\eta (D) = \mu (D)^\eta$ holds when $\eta
  \leq 1$ and $D \in \mathscr{D}$, the condition $\mathscr{M}_\mu^\eta
  (E \cap D) = \mathscr{M}_\mu^\eta (D)$ is equivalent to
  $\mathscr{M}_\mu^\eta (E \cap D) = \mu (D)^\eta$.
\end{remark}

We shall investigate some of the properties of the classes
$\mathscr{G}_\mu^\theta$. Below are our main results.

\begin{theorem} \label{the:limsup}
  Let $\mu$ be a non-atomic and locally finite Borel measure, with
  $\mu (R) = 0$. Suppose $E_n$ is a sequence of open sets such that
  for any $0 < \eta < \theta$, there is a constant $c > 0$ such that
  \[
  \liminf_{k\to\infty} \mathscr{M}_\mu^\eta (E_n \cap D) \geq c
  \mathscr{M}_\mu^\eta (D)
  \]
  holds for any $D \in \mathscr{D}$. Then $\limsup_n E_n \in
  \mathscr{G}_\mu^\theta$.
\end{theorem}

\begin{theorem} \label{the:intersection}
  Let $\mu$ be a non-atomic and locally finite Borel measure, with
  $\mu (R) = 0$. Suppose that $E_n$ is a sequence of sets in
  $\mathscr{G}_\mu^\theta$. Then $\bigcap_n E_n \in
  \mathscr{G}_\mu^\theta$.
\end{theorem}

The following theorem makes use of the so called \emph{upper coarse
  multifractal spectrum} of the measure $\mu$, which we denote by
$G_\mu$. See Section~\ref{sec:furtherdefinitions} for a definition.

\begin{theorem} \label{the:dimension}
  Let $\mu$ be a non-atomic and locally finite Borel measure, with
  $\mu (R) = 0$. Suppose that $E$ belongs to $\mathscr{G}_\mu^\theta$
  for some $0 < \theta \leq 1$ and that there is a $t > 0$ and an
  $n_0$ such that $\mu (D) \leq 2^{-tn}$ for all $D \in \mathscr{D}_n$
  with $n \geq n_0$. Then $\dimh (E \cap D) \geq \theta t$ and
  \[
  \dimh (E \cap D) \geq \theta \sup \{\, s_0 : \exists \varepsilon >
  0, \forall s\in [t,s_0], (\theta s - G_\mu (s) > \varepsilon) \,\},
  \]
  for any $D \in \mathscr{D}$ with $\mu (D) > 0$.
\end{theorem}

The proofs of the three theorems above are given in
Section~\ref{sec:propertiesproofs}.  These theorems give us the main
properties of the classes $\mathscr{G}_\mu^\theta$. In principle, one
can use Theorem~\ref{the:limsup} to determine if a limsup-set belongs
to the class $\mathscr{G}_\mu^\theta$, but this is not always
convenient in practice. In the section below, we therefore define some
inhomogeneous potentials and use them to give an alternative method to
determine is a limsup-set belongs to $\mathscr{G}_\mu^\theta$.

\subsection{Inhomogeneous potentials and energies} \label{sec:inhomogeneous}

We define the function $Q \colon \mathbb{R}^d \times \mathbb{R}^d \to
\{\emptyset \} \cup \mathscr{D} \cup \{\mathbb{R}^d\}$ by
\[
Q (x,y) = D, \qquad x \neq y,
\]
where $D \in \mathscr{D} \cup \{ \mathbb{R}^d \}$ is chosen such that
$x,y \in D$ and $D$ is minimal in sense of inclusion. If $x = y$, we
let $Q (x,y) = \emptyset$. Note that it is necessary to include the
possibility that $Q (x,y) = \mathbb{R}^d$, since otherwise $Q$ would
not always be defined; there are points $x$ and $y$ such that $\{x,
y\}$ is a subset of no $D \in \mathscr{D}$.

We use $Q$ to define some inhomogeneous potentials and energies.

\begin{definition}
  Let $\mu$ and $\nu$ be two Borel measures on $\mathbb{R}^d$ and $0
  < \theta \leq 1$. The $(\mu, \theta)$-potential of $\nu$ is the
  function $U_\mu^\theta \nu \colon \mathbb{R}^d \to \mathbb{R} \cup
  \{\infty\}$ defined by
  \[
  U_\mu^\theta \nu (x) = \int \mu (Q(x,y))^{-\theta} \, \mathrm{d} \nu
  (y),
  \]
  and the $(\mu,\theta)$-energy of $\nu$ is the quantity
  \[
  I_\mu^\theta (\nu) = \iint \mu (Q(x,y))^{-\theta} \, \mathrm{d} \nu
  (y) \mathrm{d} \nu (x) = \int U_\mu^\theta \nu \, \mathrm{d} \nu.
  \]
\end{definition}

The reason for introducing these potentials and energies is the lemma
below, which relates the classes $\mathscr{G}_\mu^\theta$ with the
energies $I_\mu^\theta$. This lemma will be our main tool when we in
various applications determine that limsup-sets belong to a class
$\mathscr{G}_\mu^\theta$.  (We call the result a lemma and not a
theorem because of its connection to what is often called Frostman's
lemma. Compare also with the related Lemmata~\ref{lem:classical} and
\ref{lem:perssonreeve}.)

\begin{lemma} \label{lem:frostman}
  Let $\mu$ be a non-atomic and locally finite Borel measure, with
  $\mu (R) = 0$. Let $E_n$ be a sequence of open sets. If $\mu_n$ are
  non-atomic Borel measures, with $\supp \mu_n \subset E_n$, that
  converge weakly to a measure $\mu$, and if for every $\eta <
  \theta$ there is a constant $c_\eta$ such that
  \[
  I_\mu^\eta (\mu_n) = \iint \mu (Q(x,y))^{-\eta} \, \mathrm{d} \mu_n
  (x) \mathrm{d} \mu_n (y) \leq c_\eta
  \]
  for all $n$, then $\limsup_n E_n \in \mathscr{G}_\mu^\theta$.
\end{lemma}

The proof of Lemma~\ref{lem:frostman} is in
Section~\ref{sec:inhomogeneousproofs}.

\subsection{Applications} \label{sec:applications}

We give below some applications of the theory presented above.

\subsubsection{Random limsup-sets} \label{sec:random}

Let $\mu$ be a Borel probability measure on $\mathbb{R}^d$. Consider
the random sequence of points $(x_n)_{n=1}^\infty$, where the points
$x_n$ are independent and distributed according to the measure
$\mu$. Let $\alpha > 0$. We are interested in the Hausdorff dimension
of the random set
\[
E_\alpha = \limsup_n B_n
\]
where $B_n = B(x_n,n^{-\alpha})$.

In Section~\ref{sec:randomproof} we will prove the following theorem.

\begin{theorem} \label{the:randomlimsup}
  Let
  \[
  s = s (\mu) = \lim_{\varepsilon \to 0^+} \sup \{\, t : G_\mu (t)
  \geq t - \varepsilon \,\}.
  \]
  Suppose that $1/\alpha \leq s$. Almost surely, $E_\alpha \in
  \mathscr{G}_\mu^\theta$ with $\theta = (\alpha s)^{-1}$ and
  \[
  \dimh E_\alpha \geq \frac{1}{\alpha} \frac{\udimh \mu}{s}.
  \]
\end{theorem}

Note that $s (\mu)$ is clearly finite since $G_\mu$ is a bounded
function.

The estimate in Theorem~\ref{the:randomlimsup} should be compared with
the result by Ekstr\"{o}m and Persson
\cite[Theorem~1]{ekstrompersson}, that almost surely
\[
\dimh E_\alpha \geq \frac{1}{\alpha} - \delta. \qquad \text{where }
\delta = \essinf_{\substack{x \sim \mu \\ \underline{d}_\mu (x) > 1 /
    \alpha}} \overline{d}_\mu (x) - \underline{d}_\mu (x).
\]
None of these results implies the other.\footnote{In the sense that
  none follows immediately from the other. Of course, any two true
  statements imply each other.} For instance, if $\delta > 0$ then
the estimate in Theorem~\ref{the:randomlimsup} is stronger if $\alpha$
is sufficiently large but may be weaker for small $\alpha$.

Note that here we treat only the case $1/\alpha \leq s$, whereas in
\cite{ekstrompersson}, any $\alpha$ was considered. Related results
can also be found in a paper by Seuret \cite{seuret}.

\subsubsection{Dynamical Diophantine approximation} \label{sec:shrinking}

Suppose that $X$ is a metric space, $T \colon X \to X$ is a map and
that $\mu$ is a $T$-invariant probability measure.

If $y \in X$ and $r_n$ is a sequence of positive numbers that
decreases to 0, then we say that the balls $B(y, r_n)$ are shrinking
targets around $y$. One is often interested in if and how often the
orbit of a point $x$ hits one of the targets, i.e.\ $T^n (x) \in B (y,
r_n)$. Such hitting properties depend very much on the points $x$ and
$y$ as well as the sequence $r_n$, but it is often possible to say
something about the behaviour for $\mu$ almost all $x$ and $y$.

Here, we assume that $T \colon I \to I$, where $I$ is a compact
interval of positive length and that $\mu$ is a $T$-invariant
probability measure such that there are positive constants $t_1$ and
$c_1$ with
\begin{equation} \label{eq:measuredecay}
  \mu (D) \leq c_1 2^{-t_1 n}
\end{equation}
whenever $D \in \mathscr{D}_n$. We assume also that $(T,\mu)$ has
summable decay of correlations for functions of bounded variation,
that is, there is a function $p$ such that
\begin{equation} \label{eq:decay}
  \biggl| \int f \circ T^n \cdot g \, \mathrm{d} \mu - \int f \,
  \mathrm{d} \mu \int g \, \mathrm{d}\mu \biggr| \leq p(n) \lVert f
  \rVert_1 \lVert g \rVert,
\end{equation}
where $\lVert g \rVert$ is the norm $\lVert g \rVert = \lVert g
\rVert_1 + \var g = \int |g| \, \mathrm{d} \mu + \var g$, and $p$
satisfies
\[
\sum_{n=0}^\infty p(n) < \infty.
\]

Given a point $x$ and a number $\alpha > 0$, we consider the set
\[
E_\alpha (x) = \{\, y : |T^n (x) - y| \leq n^{-\alpha} \text{ for
  infinitely many } n \in \mathbb{N} \,\}.
\]
This is the set of points $y$ for which the orbit of $x$ hits the
shrinking targets $B(y,n^{-\alpha})$ infinitely many times. The study
of the fractal properties of sets of this type (and other related
types) is sometimes called dynamical Diophantine approximation,
possibly after \cite{fanschmelingtroubetzkoy}.

The Hausdorff dimension of $E_\alpha (x)$ for $\mu$ almost every $x$
was given by Fan, Schmeling and Troubetzkoy in the case that $T$ is
the doubling map and $\mu$ is a Gibbs measure
\cite{fanschmelingtroubetzkoy}. Liao and Seuret extended this result
to piecewise expanding Markov maps of an interval
\cite{liaoseuret}. Some results for piecewise expanding interval maps
without a Markov structure were obtained by Persson and Rams
\cite{perssonrams}.

In Section~\ref{sec:shrinkingproofs}, we shall prove the following
result.

\begin{theorem} \label{the:shrinking}
  Let
  \[
  s = s (\mu) = \lim_{\varepsilon \to 0^+} \sup \{\, t : G_\mu (t)
  \geq t - \varepsilon \,\},
  \]
  and $\alpha$ be such that $1/\alpha \leq s$. Then, almost surely,
  $E_\alpha (x) \in \mathscr{G}_\mu^\theta$, where $\theta = (\alpha s)^{-1}$.
  
  In particular, there is a set $A \subset I$ of full $\mu$ measure
  such that if $x_1, x_2, \ldots$ is a finite or countable sequence of
  points in $A$, then
  \[
  \frac{1}{\alpha} \frac{\udimh \mu}{s} \leq \dimh \bigcap_n E_\alpha
  (x_n) \leq \frac{1}{\alpha}.
  \]
\end{theorem}

Below, we give two concrete applications of
Theorem~\ref{the:shrinking}. The proofs of these results are also in
Section~\ref{sec:shrinkingproofs}.

\subsubsection{The quadratic family} \label{sec:quadratic}

Let us consider the family of quadratic maps $T_a \colon [-1,1] \to
[-1,1]$ defined by $T_a (x) = 1 - a x^2$, where $a$ is a parameter in
$[0,2]$. Let $\gamma > 0$ be a small number and let
\[
\Delta = \{\, a \in [0, 2] : |T_a^n (0)| \geq e^{-\gamma n} \text{ and
} |(T_a^n)' (T_a(0)) | \geq 1.9^n \text{ for all } n \,\}.
\]
Benedicks and Carleson \cite{benedickscarleson} have proved that
$\Delta$ has positive Lebesgue measure. Using
Theorem~\ref{the:shrinking} we can prove the following result for
quadratic maps with $a \in \Delta$ .

\begin{corollary} \label{cor:quadratic}
  Let $T_a \colon [-1,1] \to [-1,1]$ be the quadratic map $T_a (x) = 1
  - a x^2$, where $a \in \Delta$. There is a set $A \subset [T_a^2
    (0), T_a (0)] = [1-a,1]$ of full Lebesgue measure, such that is
  $x_1, x_2, \ldots$ is a finite or countable sequence of points of\/
  $A$, then
  \[
  \dimh \bigcap_n E_\alpha (x_n) = \frac{1}{\alpha}.
  \]
\end{corollary}

\subsubsection{Piecewise expanding maps} \label{sec:piecewise}

Suppose that $T \colon [0,1] \to [0,1]$ is uniformly expanding and
piecewise $C^2$ with respect to a finite partition, and that $T$ is
covering in the sense that for any non-trivial interval $I \subset
[0,1]$, there is an $n$ such that $[0,1] \setminus W \subset T^n (I)$,
where $W$ denotes the set of points that eventually hit a
discontinuity (i.e.\ an endpoint of a partition element).

We assume that $\phi \colon [0,1] \to \mathbb{R}$ satisfies the
following conditions. There is a number $n_0$ such that
\[
\sup e^{S_{n_0} \phi} < \inf L_\phi^{n_0} 1,
\]
where $S_{n_0} \phi = \phi + \phi \circ T + \cdots + \phi \circ
T^{n_0-1}$ and
\[
L_\phi f (x) = \sum_{T(y) = x} e^{\phi (y)} f(y).
\]
Finally, $\phi$ is assumed to be piecewise $C^2$ with respect to the
same partition as $T$, and bounded from below.

Under these conditions, Liverani, Saussol and Vaienti proved that
there is a Gibbs measure $\mu_\phi$ with respect to the potential
$\phi$ and correlations decay exponentially fast for functions of
bounded variation \cite[Theorem~3.1]{liveranietal}. To the measure
$\mu_\phi$ corresponds a conformal measure $\nu_\phi$ and a density
$h_\phi$, that is
\[
\mu_\phi (E) = \int_E h_\phi \, \mathrm{d}\nu_\phi
\]
and if $E$ is a subset of a partition element, then
\[
\nu_\phi (T(E)) = \int_E e^{P(\phi) - \phi} \, \mathrm{d} \nu_\phi,
\]
where $P(\phi)$ is the topological pressure of $\phi$.  The density
$h_\phi$ is bounded and bounded away from zero. This implies that
$\underline{d}_{\mu_\phi}$ is an invariant function mod $\mu_\phi$ and
hence $\underline{d}_{\mu_\phi}$ is constant almost everywhere. We
conclude that $\dimh \mu_\phi = \ldimh \mu_\phi = \udimh \mu_\phi$.

\begin{corollary} \label{cor:piecewise}
  Let $T \colon [0,1] \to [0,1]$ be a piecewise expanding map
  satisfying the assumptions above, and assume that
  \begin{equation} \label{eq:dim-coarse}
  \dimh \mu_\phi = \lim_{\varepsilon \to 0} \sup \{\, t : G_{\mu_\phi}
  (t) \geq t - \varepsilon \,\}.
  \end{equation}
  Take $\alpha$ such that $1/\alpha \leq \dimh \mu$.
  
  There is a set $A \subset I$ of full $\mu$ measure such that if
  $x_1, x_2, \ldots$ is a finite or countable sequence of points in
  $A$, then
  \[
  \dimh \bigcap_n E_\alpha (x_n) = \frac{1}{\alpha}.
  \]
\end{corollary}

\begin{remark}
  It is known that for some Markov maps $T$, the assumption
  \eqref{eq:dim-coarse} is satisfied, see Fan--Feng--Wu
  \cite{fanfengwu}. It is unknown to me if this is known in full
  generality in the setting of Corollary~\ref{cor:piecewise}. One might
  expect that \eqref{eq:dim-coarse} is always satisfied in this case.

  Given that \eqref{eq:dim-coarse} is satisfied,
  Corollary~\ref{cor:piecewise} improves the result by Rams and myself
  \cite[Theorem~2]{perssonrams}. We proved that $\dimh E_\alpha (x) =
  \frac{1}{\alpha}$ for
  \[
  \frac{1}{\alpha} < t_1 = \limsup_{m\to \infty} \inf \frac{S_m \phi - m
    P(\phi)}{- \log |(T^m)'|}.
  \]
  whereas Theorem~\ref{eq:dim-coarse} gives the dimension for a wider
  range of $\alpha$ as well as for intersections of several sets
  $E_\alpha$. The result by Liao and Seuret \cite{liaoseuret} covers
  also the case that $\frac{1}{\alpha} > \dimh \mu$, but is only valid
  for Markov maps, and only gives the dimension for a single set
  $E_\alpha$.
\end{remark}

\section{Further definitions} \label{sec:furtherdefinitions}

In this section we give some further definitions that will be used in
this paper.

\subsection{Hausdorff net-measures}

Occasionally, we will make use of the Hausdorff net-measures
$\mathscr{N}^s$, defined by
\[
\mathscr{N}^s (E) = \lim_{\delta \to 0} \mathscr{N}_\delta^s (E),
\]
where
\[
\mathscr{N}_\delta^s (E) = \inf \biggl\{\, \sum_k |D_k|^s : D_k
\in \mathscr{D}, \ |D_k| < \delta ,\ E \subset \bigcup_k D_k
\,\biggr\}.
\]
In particular, we have $\mathscr{N}_\infty^s =
\mathscr{M}_\lambda^{\theta}$, where $\lambda$ denotes the
$d$-dimensional Lebesgue measure and $s = \theta d$. Note the
conceptual difference between the parameters $s$ and $\theta$: The
parameter $s$ should be thought of as a dimension whereas the
parameter $\theta$ should be thought of something that interpolates
between dimension 0 ($\theta = 0$) and full dimension ($\theta =
1$). We will use $s$ and $t$ to denote dimensions and $\theta$ and
$\eta$ to denote such interpolating parameters.

Falconer \cite{falconer1, falconer2} studied the classes
$\mathscr{G}_\lambda^{s/d}$, denoting them by $\mathscr{G}^s$. Our
study is a generalisation of that of Falconer. Several of the proofs
in Section~\ref{sec:propertiesproofs} are similar to the corresponding
proofs in \cite{falconer2}.

\subsection{Dimension spectra and dimension of measures} \label{ssec:dimensiondefinitions}

The {\em lower local dimension} of a measure $\mu$ at a point $x$ is
defined to be
\[
\underline{d}_\mu (x) = \liminf_{r \to 0} \frac{\log \mu
  (B(x,r))}{\log r}.
\]
Equivalently,
\[
\underline{d}_\mu (x) = \sup \{\, s : \exists C \text{ s.t. } \mu
(B(x,r)) \leq C r^s, \ \forall r \in (0, r_0 (x)) \,\}.
\]
Let
\begin{align*}
  D_\mu (s) &= \dimh \{\, x \in \supp \mu : \, \underline{d}_\mu (x)
  \leq s \,\},\\
  G_\mu (s) &= \lim_{\varepsilon \to 0} \limsup_{r \to 0} \frac{\log
    (N (s + \varepsilon, r) - N (s - \varepsilon, r))}{- \log r},
\end{align*}
where $N (s, r)$ denotes the number of $d$-dimensional cubes of
the form
\[
Q = [k_1 r, (k_1+1) r) \times \cdots \times [k_d r, (k_d+1) r)
\]
with $k_1, \ldots, k_d \in \mathbb{Z}$ and $\mu (Q) \geq r^s$. The
definition of the function $D_\mu$ is similar to what is called the
\emph{Hausdorff multifractal spectrum} of the lower local dimension of
$\mu$, which is obtained if we replace $\underline{d}_\mu (x) \leq s$
by $\underline{d}_\mu (x) = s$. The function $D_\mu$ is clearly
increasing, but this is not the case for the Hausdorff multifractal
spectrum. However, $D_\mu$ and the Hausdorff multifractal spectrum are
often the same for small values of $s$.

The function $G_\mu$ is called the \emph{upper coarse multifractal
  spectrum} of $\mu$.

We will make use of the {\em lower Hausdorff dimension} of a measure
$\mu$, which is defined as the number
\[
  \ldimh \mu = \essinf_{\mu} \underline{d}_\mu = \inf \{\, \dimh A :
  \mu (A) > 0 \,\},
\]
see for instance Falconer \cite[Proposition~10.2]{falconertechniques}.
Similarly, the {\em upper Hausdorff dimension} of $\mu$ is the number
\[
  \udimh \mu = \esssup_{\mu} \underline{d}_\mu = \inf \{\, \dimh A :
  \mu (A^\complement) = 0 \,\},
\]
which was used in the statement and will be used in the proof of
Theorem~\ref{the:randomlimsup}.

If $\ldimh \mu = \udimh \mu$, then the {\em Hausdorff dimension} of
the measure $\mu$ is defined as $\dimh \mu = \ldimh \mu = \udimh \mu$.

\section{Proofs of properties of the classes $\mathscr{G}_\mu^\theta$} \label{sec:propertiesproofs}

In this section, we will prove the Theorems~\ref{the:limsup},
\ref{the:intersection} and \ref{the:dimension}.

Throughout this section, we will assume that $\mu$ is a non-atomic and
locally finite Borel measure. We will mention explicitly when we
assume that $\mu (R) = 0$.

In order to prove Theorem~\ref{the:limsup}, we will first prove a
series of lemmata.

\begin{lemma} \label{lem:removec}
  Suppose $E \subset \mathbb{R}^d$, $0 < \theta \leq 1$ and that
  there is a constant $c > 0$ such that
  \[
  \mathscr{M}_\mu^\theta (E \cap D) \geq c \mathscr{M}_\mu^\theta (D)
  \]
  holds whenever $D \in \mathscr{D}$. Then
  \[
  \mathscr{M}_\mu^\eta (E \cap D) = \mathscr{M}_\mu^\eta (D)
  \]
  holds whenever $D \in \mathscr{D}$ and $0 < \eta < \theta$.
\end{lemma}

\begin{proof}
  Take $0 < \eta < \theta$ and $D \in \mathscr{D}$, and let
  $\{D_l\}_{l=1}^\infty$ be a disjoint cover by dyadic cubes of the
  set $E \cap D$.

  Since $\mu (D)$ is finite and $\mu$ is not atomic, we may choose a
  number $m$ such that whenever $C \subset D$ and $C \in
  \mathscr{D}_n$ for some $n \geq m$, then
  \[
  \mu (C)^{\eta-\theta} \geq c^{-1} \mu (D)^{\eta-\theta}.
  \]

  We write $D$ as a finite disjoint union of dyadic cubes $\mathscr{Q}
  = \{Q_k\}$ in the following way. All $D_k$ that belong to
  $\mathscr{D_n}$ for some $n < m$ are put in $\mathscr{Q}$. The part
  of $D$ that are not covered by these $D_k$ can be written as a union
  of elements in $\mathscr{D}_m$. These elements of $\mathscr{D}_m$
  are added to $\mathscr{Q}$. This construction implies that for any
  $Q_k \in \mathscr{Q}$, we have that one and only one of the
  following two cases holds.  \renewcommand{\theenumi}{\roman{enumi}}
  \begin{enumerate}
  \item \label{case:large}
    there exists an $l$ such that $Q_k = D_l$ and $Q_k \in
    \mathscr{D}_n$ with $n < m$.
  \item \label{case:small} 
    $Q_k \in \mathscr{D}_m$ and all cubes $D_l$ that intersect $E
    \cap Q_k$ are contained in $Q_k$, that is, each of them belong to
    some $\mathscr{D}_n$ with $n \geq m$.
  \end{enumerate}
  
  Suppose $Q_k$ satisfies (\ref{case:large}) above. Then
  \begin{equation} \label{eq:case1}
    \sum_{D_l \subset Q_k} \mu (D_l)^\eta = \mu (Q_k)^\eta = \mu
    (Q_k)^{\eta-\theta} \mu (Q_k)^\theta \geq \mu (D)^{\eta-\theta}
    \mu (Q_k)^\theta.
  \end{equation}

  Suppose that $Q_k$ satisfies (\ref{case:small}) above. By the choice
  of $m$, we have for $D_l \subset Q_k$ that
  \[
  \mu (D_l)^\eta = \mu (D_l)^{\eta-\theta} \mu (D_l)^\theta \geq
  c^{-1} \mu (D)^{\eta-\theta} \mu (D_l)^\theta.
  \]
  Hence
  \begin{align}
    \sum_{D_l \subset Q_k} \mu (D_l)^\eta &\geq \sum_{D_l \subset Q_k}
    c^{-1} \mu (D)^{\eta-\theta} \mu (D_l)^\theta \geq c^{-1} \mu
    (D)^{\eta-\theta} \mathscr{M}_\mu^\theta (E \cap Q_k) \nonumber
    \\ &\geq \mu (D)^{\eta-\theta} \mathscr{M}_\mu^\theta (Q_k) = \mu
    (D)^{\eta-\theta} \mu (Q_k)^\theta. \label{eq:case2}
  \end{align}
  Together, \eqref{eq:case1} and \eqref{eq:case2} show that
  \[
  \sum_{D_l \subset Q_k} \mu (D_l)^\eta \geq \mu (D)^{\eta-\theta} \mu
  (Q_k)^\theta
  \]
  for any $Q_k$.

  Summing over all $Q_k$, we finally get
  \begin{align*}
    \sum_l \mu (D_l)^\eta &= \sum_k \sum_{D_l \subset Q_k} \mu
    (D_l)^\eta \geq \mu (D)^{\eta-\theta} \sum_k \mu (Q_k)^\theta
    \\ &\geq \mu (D)^{\eta-\theta} \mathscr{M}_\mu^\theta (D) =
    \mathscr{M}_\mu^\eta (D). \qedhere
  \end{align*}
\end{proof}

\begin{lemma} \label{lem:open}
  Suppose $E \subset \mathbb{R}^d$, $0 < \theta \leq 1$ and that
  there is a constant $c > 0$ such that
  \[
  \mathscr{M}_\mu^\theta (E \cap D) \geq c \mathscr{M}_\mu^\theta (D)
  \]
  holds whenever $D \in \mathscr{D}$. Then
  \[
  \mathscr{M}_\mu^\theta (E \cap U) \geq c \mathscr{M}_\mu^\theta (U)
  \]
  holds for all open sets $U$.
\end{lemma}

\begin{proof}
  Write $U$ as a disjoint union of dyadic cubes from $\mathscr{D}$,
  \[
  U = \bigcup_k Q_k,
  \]
  and suppose that $(D_l)_{l=1}^\infty$ is a disjoint cover of $E \cap
  U$ by dyadic cubes.

  For each $Q_k$, with $\mu (Q_k) > 0$ the set $E \cap Q_k$ is not
  empty, since it would violate the assumption. Hence, for each $Q_k$,
  either $\mu (Q_k) = 0$ or we have one of the following two cases:
  \renewcommand{\theenumi}{\roman{enumi}}
  \begin{enumerate}
  \item \label{case:one} The dyadic cube $Q_k$ intersect exactly one
    dyadic cube $D_l$. We then have $Q_k \subset D_l$.
  \item \label{case:several} The dyadic cube $Q_k$ intersect more that
    one of the dyadic cubes $D_l$.  Then, if $D_l$ intersect $Q_k$ we
    have that $D_l$ is a strict sub-cube of $Q_k$.
  \end{enumerate}
  In case (\ref{case:several}), we have
  \[
  \sum_{D_l \subset Q_k} \mu (D_l)^\theta \geq \mathscr{M}_\mu^\theta
  (E \cap Q_k) \geq c \mathscr{M}_\mu^\theta (Q_k) = c \mu
  (Q_k)^\theta.
  \]

  We define a cover $\{C_k\}_{k=1}^\infty$ of $U$ using the covers
  $\{Q_k\}_{k=1}^\infty$ and $\{D_l\}_{l=1}^\infty$. For each $k$, if
  $Q_k$ satisfies (\ref{case:one}), let
  \[
  \mathscr{Q}_k = \{\, D_l : D_l \cap Q_k \neq \emptyset \,\}.
  \]
  Note that (\ref{case:one}) says that $\mathscr{Q}_k$ contains
  exactly one element.  Otherwise, if $Q_k$ satisfies
  (\ref{case:several}), let
  \[
  \mathscr{Q}_k = \{Q_k\}.
  \]
  Hence, in both cases, $\mathscr{Q}_k$ contains exactly one element.

  Put $\{C_k\} = \bigcup_{k=1}^\infty \mathscr{Q}_k$. We then have
  \[
  \sum_l \mu (D_l)^\theta \geq \sum_k c \mu (C_k)^\theta \geq c
  \mathscr{M}_\mu^\theta (U).
  \]
  Since $\{D_l\}_{l=1}^\infty$ is an arbitrary cover of $E \cap U$,
  this proves that $\mathscr{M}_\mu^\theta (E \cap U) \geq c
  \mathscr{M}_\mu^\theta (U)$.
\end{proof}

Recall that $R_n$ denotes the boundaries of the dyadic cubes of
$\mathscr{D}_n$.

\begin{lemma} \label{lem:opensets}
 Suppose $U$ is open and $\mu (R_n) = 0$. Then $\mathscr{M}_\mu^\theta
 (U \setminus R_n) = \mathscr{M}_\mu^\theta (U)$.
\end{lemma}

\begin{proof}
  Let $D \in \mathscr{D}$ and let $D_k$ be a cover of $D \setminus
  \partial R_n$ by dyadic cubes. Then
  \begin{equation} \label{eq:coverestimate0}
    \sum \mu (D_k)^\theta \geq \Bigl( \sum \mu(D_k) \Bigr)^\theta \geq
    \mu (D)^\theta,
  \end{equation}
  since $\mu (D \setminus R_n) = \mu (D)$, which shows that
  \begin{equation} \label{eq:coverestimate}
    \mathscr{M}_\mu^\theta (D \setminus R_n) = \mathscr{M}_\mu^\theta
    (D).
  \end{equation}

  Now, let $D_k$ be a disjoint dyadic cover of $U \setminus R_n$. We
  will modify $\{D_k\}$ into a cover $\mathscr{C} = \{C_k\}$ of $U$
  such that
  \[
  \sum \mu (D_k)^\theta \geq \sum \mu (C_k)^\theta.
  \]
  
  If $D_k \in \mathscr{D}_m$ for some $m \leq n$, then we put $D_k$ in
  $\mathscr{C}$. 

  We now proceed by induction over $m > n$. Start with $m = n +1$ and
  suppose that $C \in \mathscr{D}_m$. If $C$ is covered by the
  elements already in $\mathscr{C}$, then we do nothing. Otherwise we
  proceed as follows.

  If $C \setminus R_n$ is covered by the cubes in
  \[
  \mathscr{E} (C) = \biggl\{\, D_k : D_k \in \bigcup_{l \geq m}
  \mathscr{D}_l,\ D_k \subset C\, \biggr\},
  \]
  then
  \begin{equation} \label{eq:coverestimate2}
    \sum_{D_k \in \mathscr{E} (C)} \mu (D_k)^\theta \geq \mu
    (C)^\theta
  \end{equation}
  by \eqref{eq:coverestimate0} (or by \eqref{eq:coverestimate}), and
  we let $C$ be an element of $\mathscr{C}$.

  This process is carried out for each $C \in \mathscr{D}_m$. The
  number $m$ is then incremented by one and the process is repeated.

  In this way we get the new cover $\mathscr{C}$, which by
  construction is disjoint. By \eqref{eq:coverestimate2} we have
  \begin{equation} \label{eq:coverestimate3}
    \sum \mu (D_k)^\theta \geq \sum_{C \in \mathscr{C}} \mu (C)^\theta.
  \end{equation}

  It is clear that $\mathscr{C}$ covers $U \setminus R_n$, since
  $\bigcup_k D_k \subset \bigcup_{C \in \mathscr{C}} C$. Let $x \in R_n
  \cap U$. Then there is a largest dyadic cube $D$ such that $x \in D
  \subset U$. By the construction of $\mathscr{C}$ we have that either
  $D \in \mathscr{C}$, or there is a $C \in \mathscr{C}$, with $D
  \subset C$. In any case, the point $x$ is covered by $\mathscr{C}$
  and since $x \in R_n \cap U$ is arbitrary, this proves that
  $\mathscr{C}$ covers $U$. Hence
  \begin{equation} \label{eq:coversU}
  \sum_{C \in \mathscr{C}} \mu (C)^\theta \geq \mathscr{M}_\mu^\theta
  (U).
  \end{equation}

  By \eqref{eq:coverestimate3} and \eqref{eq:coversU}, we have
  $\mathscr{M}_\mu^\theta (U \setminus R_n) \geq
  \mathscr{M}_\mu^\theta (U)$, so that $\mathscr{M}_\mu^\theta (U
  \setminus R_n) = \mathscr{M}_\mu^\theta (U)$.
\end{proof}

The following increasing sets lemma is a special case of Theorem~52 of
Rogers \cite{rogers}. It is needed to prove
Lemma~\ref{lem:intersection} and Theorem~\ref{the:dimension}.

\begin{lemma} \label{lem:increasingsets}
  If $F_1 \subset F_2 \subset \cdots$ is an increasing sequence of\/
  subsets of\/ $\mathbb{R}^d$, then
  \[
    \mathscr{M}_\mu^\theta \biggl( \bigcup_{n=1}^\infty F_n \biggr) =
    \sup_n \mathscr{M}_\mu^\theta (F_n).
  \]
\end{lemma}

\begin{lemma} \label{lem:intersection}
  Suppose $\mu (R) = 0$ and that $\{E_n\}_{n=1}^\infty$ is a sequence
  of open sets such that
  \begin{equation} \label{eq:intersectionlemma}
    \mathscr{M}_\mu^\theta (E_n \cap U) =
    \mathscr{M}_\mu^\theta (U)
  \end{equation}
  holds for any $n$ and any open set $U$. Then, for any open set $U$,
  we have
  \[
  \mathscr{M}_\mu^\theta \biggl(U \cap \bigcap_{n=1}^\infty E_n
  \biggr) = \mathscr{M}_\mu^\theta (U).
  \]
\end{lemma}

\begin{proof}
  If $V$ is a set, we define
  \[
  V_{(-\delta)} = \{\, x \in V : \inf_{y\not \in V} |x-y| > \delta
  \,\}.
  \]

  Let $\varepsilon > 0$. We define a sequence of open sets $U_n$
  inductively. Put $U_0 = U$. Suppose that $U_{n-1}$ has been defined
  and satisfies
  \[
  \mathscr{M}_\mu^\theta (U_{n-1}) > \mathscr{M}_\mu^\theta (U) -
  \varepsilon.
  \]
  Put $U_n = (E_n \cap (U_{n-1} \setminus (R_n \cup R_{-n})
  ))_{(-\delta_n)}$. Then, by Lemma~\ref{lem:increasingsets}, we can
  choose $\delta_n$ so small that $\mathscr{M}_\mu^\theta (U_n)$ is as
  close to $\mathscr{M}_\mu^\theta (E_n \cap (U_{n-1} \setminus (R_n
  \cup R_{-n})))$ as we desire. But
  \begin{align*}
    \mathscr{M}_\mu^\theta (E_n \cap (U_{n-1} \setminus (R_n \cup
    R_{-n}))) & = \mathscr{M}_\mu^\theta (U_{n-1} \setminus (R_n \cup
    R_{-n})) \\ & =\mathscr{M}_\mu^\theta (U_{n-1}) >
    \mathscr{M}_\mu^\theta (U) - \varepsilon,
  \end{align*}
  holds by \eqref{eq:intersectionlemma}, Lemma~\ref{lem:opensets} and
  the assumption on $U_{k-1}$. Hence we can choose $\delta_n$ so small
  that
  \[
  \mathscr{M}_\mu^\theta (U_n) > \mathscr{M}_\mu^\theta (U) -
  \varepsilon.
  \]
  Hence by induction, we obtain a sequence $\{U_n\}_{n=0}^\infty$ such
  that
  \[
  \overline{U}_n \subset U \cap E_n \setminus (R_n \cup R_{-n}),
  \qquad \text{and} \qquad \mathscr{M}_\mu^\theta (U_n) >
  \mathscr{M}_\mu^\theta (U) - \varepsilon,
  \]
  holds for every $n$, from which it follows that
  \[
  \bigcap_{n=0}^\infty \overline{U}_n \subset U \cap
  \bigcap_{n=1}^\infty E_n \setminus R, \qquad \text{and} \qquad
  \mathscr{M}_\mu^\theta (U_n) > \mathscr{M}_\mu^\theta (U) -
  \varepsilon,
  \]
  holds for every $n$.

  Let $\{D_k\}_{k=1}^\infty$ be a cover of $U \cap
  \bigcap_{n=1}^\infty E_n$ by dyadic cubes. Let $V = \bigcup_k (D_k
  \setminus \partial D_k)$. The compact set $\bigcap_n \overline{U}_n$
  is contained in the open set $V$. Hence there exists an $n$ such
  that
  \[
  \overline{U}_n \subset V \subset \bigcup_{k = 1}^\infty D_k.
  \]
  Since
  \[
  \sum_{k=1}^\infty \mu (D_k)^\theta \geq \mathscr{M}_\mu^\theta
  (\overline{U}_n) \geq \mathscr{M}_\mu^\theta (U_n) >
  \mathscr{M}_\mu^\theta (U) - \varepsilon,
  \]
  this proves that
  \[
  \mathscr{M}_\mu^\theta \biggl(U \cap \bigcap_{n=1}^\infty E_n
  \biggr) \geq \mathscr{M}_\mu^\theta (U) - \varepsilon.
  \]
  As $\varepsilon > 0$ is arbitrary, this implies that
  \[
  \mathscr{M}_\mu^\theta \biggl(U \cap \bigcap_{n=1}^\infty E_n
  \biggr) = \mathscr{M}_\mu^\theta (U). \qedhere
  \]
\end{proof}

We can now prove Theorem~\ref{the:limsup} and \ref{the:intersection}.

\begin{proof}[Proof of Theorem~\ref{the:limsup}] 
  Clearly, $\limsup_n E_n$ is a $G_\delta$-set, so we only need to show
  that
  \[
  \mathscr{M}_\mu^\eta (\limsup_n E_n \cap D) \geq \mathscr{M}_\mu^\eta
  (D)
  \]
  holds for any $\eta < \theta$ and $D \in \mathscr{D}$.

  Let $F_n = \bigcup_{k > n} E_k$. Then, by Lemma~\ref{lem:removec},
  \[
  \mathscr{M}_\mu^\eta (F_n \cap D) = \mathscr{M}_\mu^\eta (D)
  \]
  holds for any $\eta < \theta$ and any $D \in \mathscr{D}$. Hence, by
  Lemma~\ref{lem:open}
  \[
  \mathscr{M}_\mu^\eta (F_n \cap U) \geq \mathscr{M}_\mu^\eta (U)
  \]
  holds for any $\eta < \theta$ and any open set $U$. Now,
  Lemma~\ref{lem:intersection} finishes the proof.
\end{proof}

\begin{proof}[Proof of Theorem~\ref{the:intersection}]
  Since each $E_n$ is a $G_\delta$-set, we can write $E_n$ as an
  intersection
  \[
  E_n = \bigcap_{k=1}^\infty U_{n,k},
  \]
  where each $U_{n,k}$ is an open set.

  Then $\mathscr{M}_\mu^\eta (D \cap U_{n,k}) = \mathscr{M}_\mu^\eta
  (D)$ holds for any $U_{n,k}$, $\eta < \theta$ and $D \in
  \mathscr{D}$. By Lemma~\ref{lem:open}, we have that
  $\mathscr{M}_\mu^\eta (U \cap U_{n,k}) = \mathscr{M}_\mu^\eta (U)$
  holds for any $U_{n,k}$, $\eta < \theta$ and open $U$.

  Now, Lemma~\ref{lem:intersection} implies that
  \[
  E = \bigcap_{n=1}^\infty \bigcap_{k=1}^\infty U_{n,k}
  \]
  satisfies $\mathscr{M}_\mu^\eta (E \cap U) = \mathscr{M}_\mu^\eta
  (U)$ for any open set and $\eta < \theta$. Since $E$ is a
  $G_\delta$-set this shows that $E$ belongs to
  $\mathscr{G}_\mu^\theta$.
\end{proof}

We end this section with the proof of Theorem~\ref{the:dimension}, but
before we do so, we give a remark on possible improvements.

\begin{remark} \label{rem:dimension}
  Put
  \[
  \Theta_n (s) = \bigcap_{k=n}^\infty \{\, x : \mu (B (x,r)) < r^s
    \text{ for } r = \sqrt{d} 2^{-k} \,\},
  \]
  and
  \begin{align*}
  \Theta (s) &= \bigcup_{n=1}^\infty \Theta_n (s)
  \\ &= \{\, x : \mu (B (x,r)) < r^s \text{ for all } r = \sqrt{d}
  2^{-n},\ n \in \mathbb{N},\ n> n_0 (x) \,\},
  \end{align*}

  Let $s < \ldimh \mu$ and $0 < \eta < \theta$.  Since $s < \ldimh \mu
  = \essinf \underline{d}_\mu$ we have $\mu (\Theta(s)^\complement) = 0$.  It
  follows that if $D \in \mathscr{D}$, then
  \[
  \mathscr{M}_\mu^\eta (D \cap \Theta (s)) =
  \mathscr{M}_\mu^\eta (D).
  \]
  The set $\Theta (s)$ is a $G_{\delta \sigma}$ set. If it were a
  $G_\delta$ set, we would conclude that $\Theta (s) \in
  \mathscr{G}_\mu^\theta$, and so $\Theta (s) \cap E \in
  \mathscr{G}_\mu^\theta$ whenever $E \in
  \mathscr{G}_\mu^\theta$. This would mean that
  \[
  \mathscr{M}_\mu^\eta (D \cap \Theta (s) \cap E) \geq
  \mathscr{M}_\mu^\eta (D) > 0
  \]
  if $\mu (D) > 0$. The increasing sets lemma would then give that
  \[
  \mathscr{M}_\mu^\eta (D \cap \Theta_{n_0} (s) \cap E) > 0
  \]
  if $n_0$ is large enough, and this in turns easily leads to the
  estimate that $\dimh (D \cap E) \geq \theta \ldimh \mu$, which is a
  stronger estimate than that in Theorem~\ref{the:dimension}. Indeed,
  if $\{D_k\}$ is a dyadic cover of $D \cap \Theta_{n_0} (s) \cap E$ such
  that each $D_k$ belongs to some $\mathscr{D}_n$ with $n \geq n_0$,
  then
  \[
  \sum |D_k|^{\eta s} \geq \sum \mu (D_k)^\eta \geq
  \mathscr{M}_\mu^\eta (D \cap \Theta_{n_0} (s) \cap E) > 0,
  \]
  which shows that $\dimh (D \cap \Theta_{n_0} (s) \cap E) \geq \eta
  s$.

  Even if $\Theta (s)$ is not a $G_\delta$ set, I still suspect that
  \[
  \mathscr{M}_\mu^\eta (D \cap \Theta_n (s) \cap E) > 0
  \]
  holds when $\mu (D) > 0$, but I have not been able to prove this. It
  would be interesting to know whether this is true or not, since if
  true, it leads to a stronger dimension estimate than that of
  Theorem~\ref{the:dimension}. More generally, is it true that if
  $\Gamma$ is a $G_{\delta \sigma}$ set of full $\mu$ measure and $E
  \in \mathscr{G}_\mu^\theta$, then $\mathscr{M}_\mu^\eta (\Gamma \cap
  E) > 0$? Is it enough that $\Gamma$ is a $G_{\delta \sigma}$ set of
  positive measure?

  In any case, it seems not to be a great disadvantage to only have
  the weaker dimension estimate of Theorem~\ref{the:dimension}, since
  in all applications (Section~\ref{sec:applications}) we are actually
  able to prove not only that the set $E$ under consideration belongs
  to $\mathscr{G}_\mu^\theta$, but also to $\mathscr{G}_\nu^\theta$,
  where $\nu = \mu |_{\Theta_n (s)}$ for $s < \udimh \mu$ and $n$ such
  that $\mu (\Theta_n (s)) > 0$. This is sufficient to conclude that
  $\dimh E \geq \theta \udimh \mu$.
\end{remark}

\begin{proof}[Proof of Theorem~\ref{the:dimension}]
  The first estimate of $\dimh (E \cap D)$ is trivial. Suppose $\mu
  (D) > 0$ and let $\{D_k\}$ be a dyadic cover of $E \cap D$. Then
  \[
  \sum |D_k|^{\theta t} \geq c \sum \mu (D_k)^\theta \geq c
  \mathscr{M}_{\mu}^\theta (E \cap D) > 0,
  \]
  which shows that $\mathscr{N}^{\theta t} (E \cap D) > 0$ (the
  Hausdorff net-measure) and hence $\dimh (E \cap D) \geq \theta t$.

  Take $\eta < \theta$. Let $s_0$ and $\varepsilon_0$ be such that
  $\theta s - G_\mu (s) > \varepsilon_0 > 0$ for all $s \in [t,s_0]$,
  and $D$ such that $\mu (D) > 0$. It is sufficient to prove that
  $\dimh (E \cap D) \geq s_0 \theta$.

  Take $\varepsilon > 0$ and $\delta_0 > 0$. There is then a $\delta =
  \delta (s) > 0$ such that $\delta < \delta_0$ and
  \[
  \limsup_{r \to 0} \frac{\log (N(s+\delta, r) - N(s-\delta, r))}{-
    \log r} \leq G_\mu (s) + \frac{1}{2} \varepsilon,
  \]
  and hence also an $r_0 = r_0 (s) > 0$ such that
  \[
  \frac{\log (N(s+\delta, r) - N(s-\delta, r))}{-
    \log r} \leq G_\mu (s) + \varepsilon, \qquad r < r_0.
  \]
  Hence we have
  \begin{equation} \label{eq:numberofcubes}
    N(s+\delta, r) - N(s-\delta, r) \leq r^{G_\mu (s) + \varepsilon}. 
  \end{equation}

  By compactness, there are $t \leq s_1 < s_2 < \cdots < s_n \leq s_0$
  such that the balls $B(s_k, \delta(s_k))$, $k = 1, 2, \ldots, n$
  cover $[t,d]$. Let $\delta = \max \{\delta (s_1), \delta(s_2),
  \ldots, \delta(s_n)\}$ and $r_0 = \min \{r_0 (s_1), r_0(s_2),
  \ldots, r_0(s_n)\}$. Then $\delta < \delta_0$ since all $\delta (s)
  < \delta_0$.

  By choosing $\varepsilon$ and $\delta_0$ sufficiently small, we can
  achieve that
  \[
  \varepsilon + \theta \delta < \frac{1}{2} \varepsilon_0,
  \]
  where $\varepsilon_0$ has been chosen above to satisfy $\theta s -
  G_\mu (s) > \varepsilon_0$ for $s \in [t,s_0]$.
  
  The set $E \cap D$ is a bounded set with $\mathscr{M}_\mu^\eta (E
  \cap D) > 0$. Let $\{D_k\}$ be a cover of $E \cap D$ by dyadic cubes
  belonging to $\bigcup_{l = L}^\infty \mathscr{D}_l$, where we will
  take $L$ to be large. We group the dyadic cubes covering $E \cap D$
  according to their sizes, letting
  \[
  \mathscr{C}_l = \mathscr{D}_l \cap \{D_k\},
  \]
  each $\mathscr{C}_l$ being finite. Each $\mathscr{C}_l$ is then
  further divided into collections $\mathscr{C}_l^+ (k)$ and
  $\mathscr{C}_l^-$, defined by
  \[
    \mathscr{C}_l^+ (k) = \biggl\{\, C \in \mathscr{C}_l \setminus
    \bigcup_{j < k} \mathscr{C}_l^+ (j) : 2^{-(s_k+\delta (s_k))l} \leq
    \mu (C) \leq 2^{-(s_k-\delta (s_k))l} \,\biggr\},
  \]
  for $k = 1,2,\ldots,n$, and
  \[
  \mathscr{C}_l^- = \mathscr{C}_l \setminus \bigcup_{k=1}^n
  \mathscr{C}_l^+ (k).
  \]
  The condition $C \in \mathscr{C}_l \setminus \bigcup_{j < k}
  \mathscr{C}_l^+ (j)$ in the definition of $\mathscr{C}_l^+ (k)$ is
  there to ensure that every $C$ is in at most one $\mathscr{C}_l^+
  (k)$. Since the balls $B (s_k, \delta(s_k))$ that cover $[t,d]$ will
  necessarily overlap, it could otherwise be possible that some $C$
  belongs to more than one $\mathscr{C}_l^+ (k)$.

  By \eqref{eq:numberofcubes}, we have
  \[
  \# \mathscr{C}_l^+ (k) \leq 2^{(G_{\mu} (s_k) + \varepsilon)l}.
  \]
  This implies that
  \begin{align*}
    \sum_{l=L}^\infty \sum_{k=1}^n \sum_{C \in \mathscr{C}_l^+ (k)}
    \mu (C)^\eta &\leq \sum_{l=L}^\infty \sum_{k=1}^n 2^{(G_{\mu}
      (s_k) + \varepsilon)l} 2^{-\eta (s_k-\delta) l} \\ & =
    \sum_{l=L}^\infty \sum_{k=1}^n 2^{(G_\mu (s_k) - \eta s_k +
      \varepsilon + \eta \delta) l} \\ &\leq \sum_{l=L}^\infty n
    2^{-\frac{1}{2} \varepsilon_0 l} \leq \frac{1}{2}
    \mathscr{M}_\mu^\eta (E \cap D),
  \end{align*}
  if $L$ is large enough.

  If $C \in \mathscr{C}_l^-$, then $\mu (C) < 2^{-(s_n+\delta)l} <
  2^{-(s_0-\delta)l}$. Hence, with $s = s_0 - \delta$, we have
  \begin{multline*}
    \sum_k |D_k|^{s \eta} \geq \sum_l \sum_{C \in \mathscr{C}_l^-}
    |C|^{s \eta} \geq \sum_l \sum_{C \in \mathscr{C}_l^-} \mu
    (C)^\eta \\ = \sum_k \mu(D_k)^\eta - \sum_l \sum_k \sum_{C \in
      \mathscr{C}_l^+ (k)} \mu (C)^\eta \geq \frac{1}{2}
    \mathscr{M}_\mu^\eta (E \cap D).
  \end{multline*}
  This proves that $\mathscr{N}^{s \eta} (E \cap D) > 0$. Hence
  $\mathscr{H}^{s \eta} (E \cap D) > 0$ and $\dimh (E \cap D) \geq s
  \eta = (s_0 - \delta) \eta$. As $\delta$ can be made as small as we
  please and $\eta$ as close to $\theta$ as we like, this finishes the
  proof.
\end{proof}

\section{Proof of Lemma~\ref{lem:frostman}} \label{sec:inhomogeneousproofs}

The goal of this section is to prove Lemma~\ref{lem:frostman}.

Throughout this section, we assume that $\mu$ is a non-atomic and
locally finite Borel measure, with $\mu (R) = 0$. We start by proving
the following lemma, which is an important step towards the proof of
Lemma~\ref{lem:frostman}.

\begin{lemma} \label{lem:frostman2}
  Suppose that $E_n$ are open sets, and $\mu_n$ are Borel measures
  with $\supp \mu_n \subset E_n$. If there is a constant $C > 0$, such
  that for any $D \in \mathscr{D}$ holds
  \begin{equation} \label{eq:energycondition}
    \liminf_{n \to \infty} \int_D \biggl(\int_D \mu (Q
    (x,y))^{-\theta} \, \mathrm{d} \mu_n (y) \biggr)^{-1} \,
    \mathrm{d} \mu_n (x) \geq \frac{\mu (D)^\theta}{C},
  \end{equation}
  then $\limsup_n E_n \in \mathscr{G}_\mu^\theta$.
\end{lemma}

\begin{proof}
  Let $D \in \mathscr{D}$ and suppose that $\mu (D) > 0$. We consider
  \[
  K_\theta (x) = U_\mu^\theta (\mu_n|_D) (x) = \int_D \mu
  (Q(x,y))^{-\theta} \, \mathrm{d} \mu_n (y).
  \]
  We let $\nu$ be a probability measure with support in $D$ and
  defined by
  \[
  \nu (A) = \frac{\int_A K_\theta^{-1} \, \mathrm{d} \mu_n}{\int_D
    K_\theta^{-1} \, \mathrm{d} \mu_n}, \qquad A \subset D.
  \]

  We prove that if $n$ is large enough, then, for any $A \subset D$
  with $A \in \mathscr{D}$, we have
  \begin{equation} \label{eq:measureineq}
    \nu (A) \leq 2 C \biggl( \frac{\mu (A)}{\mu (D)} \biggr)^\theta.
  \end{equation}

  If $n$ is large enough, then the denominator in the expression
  defining $\nu$ satisfies
  \[
  \int_D K_\theta^{-1} \, \mathrm{d} \mu_n \geq
  \frac{\mu(D)^\theta}{2C}
  \]
  by \eqref{eq:energycondition}.

  By Jensen's inequality, we have (compare with the proof of
  Lemma~\ref{lem:classical})
  \begin{align*}
    \int_A K_\theta^{-1} \, \mathrm{d} \mu_n & = \int_A \biggl( \int_D
    \mu (Q(x,y))^{-\theta} \, \mathrm{d} \mu_n (y) \biggr)^{-1} \,
    \mathrm{d} \mu_n (x) \\ &\leq \int_A \biggl( \int_A \mu
    (Q(x,y))^{-\theta} \, \frac{\mathrm{d} \mu_n (y)}{\mu_n (A)}
    \biggr)^{-1} \, \frac{\mathrm{d} \mu_n (x)}{\mu_n (A)} \\ &\leq
    \int_A \biggl( \int_A \mu (Q(x,y))^{\theta} \, \frac{\mathrm{d}
      \mu_n (y)}{\mu_n (A)} \biggr) \, \frac{\mathrm{d} \mu_n
      (x)}{\mu_n (A)} \\ &\leq \mu (A)^\theta.
  \end{align*}
  This shows that \eqref{eq:measureineq} holds when $n$ is large.

  Suppose now that $\{D_k\}$ is a cover of $E_n \cap D$ by dyadic
  cubes and that $n$ is large. Then
  \[
  1 = \nu \biggl( \bigcup_k D_k \biggr) \leq \sum_k \nu (D_k) = \sum_k
  \nu (D_k \cap D) \leq \sum_k 2 C \biggl(\frac{\mu (D_k \cap
    D)}{\mu (D)} \biggr)^\theta.
  \]
  Hence, if all $D_k$ are subsets of $D$, then
  \[
  \sum_k \mu(D_k)^\theta \geq \frac{\mu (D)^\theta}{2 C},
  \]
  and so $\mathscr{M}_\mu^\theta (E_n \cap D) \geq (2C)^{-1} \mu
  (D)^\theta$.

  If $D \in \mathscr{D}$ with $\mu (D) = 0$, then
  $\mathscr{M}_\mu^\theta (E_n \cap D) \geq (2C)^{-1} \mu (D)^\theta$ is
  trivially satisfied.

  We have thus showed that
  \[
  \liminf_{n \to \infty} \mathscr{M}_\mu^\theta (E_n \cap D) \geq
  (2C)^{-1} \mu (D)^\theta, \qquad D \in \mathscr{D}.
  \]
  Now, Theorem~\ref{the:limsup} finishes the proof.
\end{proof}

\begin{lemma} \label{lem:energyestimate}
  If $\mu$ is a finite and non-atomic Borel measure and $D\in
  \mathscr{D}$, then
  \[
  \mu(D)^{1-\theta} \leq \int_D \mu (Q(x,y))^{-\theta} \, \mathrm{d}
  \mu (y) \leq \frac{\mu(D)^{1 - \theta}}{1 - \theta},
  \]
  for $x \in D$, and in particular
  \[
  (1 - \theta) \mu (D)^\theta \leq \int_D \biggl( \int_D \mu
  (Q(x,y))^{-\theta} \, \mathrm{d} \mu (y) \biggr)^{-1} \, \mathrm{d}
  \mu (x) \leq \mu (D)^\theta,
  \]
  and
  \[
  \mu(D)^{2-\theta} \leq \int_D \int_D \mu(Q(x,y))^{-\theta} \,
  \mathrm{d} \mu(x) \mathrm{d} \mu (y) \leq \frac{\mu(D)^{2-\theta}}{1
    - \theta}.
  \]
\end{lemma}

\begin{proof}
  We may assume that $\mu$ is a probability measure on $D$. By
  Fubini's theorem we can write
  \[
  \int_D \mu (Q(x,y))^{-\theta} \, \mathrm{d} \mu (y) = 1 +
  \int_1^\infty \mu (D_{n(u)}) \, \mathrm{d} u,
  \]
  where $D_{n(u)} \in \mathscr{D}_{n(u)}$ is chosen such that $x \in
  D_{n(u)}$ and $\mu (D_{n(u)})^{-\theta} \geq u$, and $n (u)$ is
  chosen as small as possible. Hence we have
  \begin{align*}
    1 \leq \int_D \mu (Q(x,y))^{-\theta} \, \mathrm{d} \mu (y) &= 1 +
    \int_1^\infty \mu (D_{n(u)}) \, \mathrm{d} u \\ &\leq 1 +
    \int_1^\infty u^{-1/\theta} \, \mathrm{d} u = \frac{1}{1 -
      \theta}. \qedhere
  \end{align*}
\end{proof}

The following two results are variants of Lemma~2.2 and Corollary~2.3
of \cite{perssonreeve}, and proved in the same way.

\begin{lemma} \label{lem:Mm}
  Let $\mu$ be a Borel measure. Assume that for some $\theta \in (0,1)$,
  \[
  I_\mu^\theta (\mu) = \iint \mu (Q(x,y))^{-\theta} \,
  \mathrm{d}\mu (x) \mathrm{d} \mu (y) < \infty.
  \]
  Then, if $M_m = \{\, (x,y) : \mu (Q(x,y))^{-\theta} > m \,\}$, then
  we have, for $0 < \eta < \theta$, that
  \[
  \iint_{M_m} \mu (Q(x,y))^{-\eta} \, \mathrm{d} \mu (x) \mathrm{d}
  \mu (y) \leq I_\mu^\theta (\mu) \frac{\theta}{\theta - \eta}
  m^{\eta/\theta - 1}.
  \]
\end{lemma}

\begin{proof}
  The assumption implies that $\mu \times \mu (M_m) \leq I_\mu^\theta
  (\mu) /m$. It then follows that
  \begin{align*}
    \iint_{M_m} \mu (Q)^{-\eta} \, \mathrm{d}\mu \mathrm{d} \mu & =
    m^{\eta/\theta} \mu \times \mu (M_m) +
    \int_{m^{\eta/\theta}}^\infty \mu \times \mu (M_{u^{\theta/\eta}})
    \, \mathrm{d} u \\ &\leq I_\mu^\theta (\mu) m^{\eta/\theta - 1} +
    I_\mu^\theta (\mu) \int_{m^{\eta/\theta}}^\infty u^{-\theta/\eta}
    \, \mathrm{d}u \\ &= I_\mu^\theta (\mu) \frac{\theta}{\theta -
      \eta} m^{\eta/\theta - 1}. \qedhere
  \end{align*}
\end{proof}

\begin{corollary} \label{cor:energyconvergence}
  If $\mu_n$ are non-atomic Borel measures that converge weakly to a
  measure $\mu$, and if $\iint \mu (Q(x,y))^{-\theta} \, \mathrm{d}
  \mu_n (x) \mathrm{d} \mu_n (y)$ are uniformly bounded for some
  $\theta \in (0,1)$, then, for $0 < \eta < \theta$ and $D \in
  \mathscr{D}$ with $\mu (D) > 0$,
  \[
  \iint_{D \times D} \mu (Q(x,y))^{-\eta} \, \mathrm{d} \mu_n (x)
  \mathrm{d} \mu_n (y) \to \iint_{D \times D} \mu (Q(x,y))^{-\eta} \,
  \mathrm{d} \mu (x) \mathrm{d} \mu (y),
  \]
  as $n \to \infty$.
\end{corollary}

\begin{proof}
  Let $\varepsilon > 0$, $0 < \eta < \theta$ and $D \in
  \mathscr{D}$ with $\mu (D) > 0$. We put
  \[
  M_m = \{\, (x,y) \in D \times D: \mu (Q(x,y))^{-\theta} > m \,\}.
  \]
  Since $\iint \mu (Q(x,y))^{-\theta} \, \mathrm{d} \mu_n (x)
  \mathrm{d} \mu_n (y)$ are uniformly bounded, the estimates of
  \[
  \iint_{M_m} \mu (Q (x,y))^{-\eta} \, \mathrm{d} \mu_n (x) \mathrm{d}
  \mu_n (y)
  \]
  provided by Lemma~\ref{lem:Mm} are uniform in $n$, and we can take $m$
  and $N$ so that
  \[
  \iint_{M_m} \mu (Q (x,y))^{-\eta} \, \mathrm{d} \mu_n (x) \mathrm{d}
  \mu_n (y) < \varepsilon
  \]
  holds for all $n > N$. We then have
  \begin{multline*}
    \iint_{D \times D} \mu (Q (x,y))^{-\eta} \, \mathrm{d}\mu_n (x)
    \mathrm{d}\mu_n (y) \\ \leq \varepsilon + \iint_{D \times D} \min \{
    \mu(Q (x,y))^{-\eta}, m^{\eta/\theta} \} \, \mathrm{d}\mu_n (x)
    \mathrm{d} \mu_n (y).
  \end{multline*}
  Since the measures $\mu_n$ converge weakly to $\mu$ and $\mu
  (\partial C) = 0$ for all $C \in \mathscr{D}$ we have that $\mu_n
  (C) \to \mu (C)$ for all $C \in \mathscr{D}$. Hence
  \begin{align*}
    \iint_{D \times D} \min \{ \mu(&Q (x,y))^{-\eta}, m^{\eta/\theta} \} \,
    \mathrm{d}\mu_n (x) \mathrm{d} \mu_n (y) \\ &\to \iint_{D \times D} \min \{
    \mu(Q (x,y))^{-\eta}, m^{\eta/\theta} \} \, \mathrm{d}\mu (x) \mathrm{d} \mu (y)
    \\ &\leq \iint_{D \times D} \mu(Q (x,y))^{-\eta} \, \mathrm{d} \mu (x)
    \mathrm{d} \mu (y),
  \end{align*}
  where the convergence holds because $\min \{ \cdots \}$ is a bounded
  function.  As $\varepsilon$ is arbitrary, this shows that
  \begin{multline*}
    \limsup_{n\to \infty} \iint_{D \times D} \mu (Q (x,y))^{-\eta} \,
    \mathrm{d} \mu_n (x) \mathrm{d} \mu_n (y) \\ \leq \iint_{D \times
      D} \mu(Q (x,y))^{-\eta} \, \mathrm{d}\mu (x) \mathrm{d} \mu (y).
  \end{multline*}

  The inequality
  \begin{multline*}
    \liminf_{n\to \infty} \iint_{D \times D} \mu (Q (x,y))^{-\eta} \,
    \mathrm{d} \mu_n (x) \mathrm{d} \mu_n (y) \\ \geq \iint_{D \times
      D} \mu(Q (x,y))^{-\eta} \, \mathrm{d}\mu (x) \mathrm{d} \mu (y)
  \end{multline*}
  is trivial since
  \begin{align*}
  \iint_{D \times D} & \mu (Q (x,y))^{-\eta} \, \mathrm{d} \mu_n (x)
  \mathrm{d} \mu_n (y) \\ & \geq \iint_{D \times D} \min \{ \mu (Q
  (x,y))^{-\eta}, m \} \, \mathrm{d} \mu_n (x) \mathrm{d} \mu_n (y)
  \\ &\to \iint_{D \times D} \min \{ \mu (Q (x,y))^{-\eta}, m \} \,
  \mathrm{d} \mu (x) \mathrm{d} \mu (y), && n \to \infty \\ &\to
  \iint_{D \times D} \mu (Q (x,y))^{-\eta} \, \mathrm{d} \mu (x)
  \mathrm{d} \mu (y), && m \to \infty,
  \end{align*}
  and finishes the proof.
\end{proof}

We can now give the proof of Lemma~\ref{lem:frostman}.

\begin{proof}[Proof of Lemma~\ref{lem:frostman}]
  By Corollary~\ref{cor:energyconvergence} we have for any $D \in
  \mathscr{D}$ that
  \[
  \iint\limits_{D \times D} \mu (Q(x,y))^{-\eta} \, \mathrm{d} \mu_n
  (x) \mathrm{d} \mu_n (y) \to \iint\limits_{D \times D} \mu
  (Q(x,y))^{-\eta} \, \mathrm{d} \mu (x) \mathrm{d} \mu (y) \leq
  c_\eta,
  \]
  as $n \to \infty$.

  Since
  \begin{multline*}
    \int_D \biggl( \int_D \mu (Q(x,y))^{-\eta} \, \mathrm{d} \mu_n
    \biggr)^{-1} \, \mathrm{d} \mu_n (y) \\ \geq \biggl( \int_D \int_D
    \mu(Q(x,y))^{-\eta} \, \frac{\mathrm{d} \mu_n (x)}{\mu_n (D)}
    \frac{\mathrm{d} \mu_n (y)}{\mu_n (D)} \biggr)^{-1},
  \end{multline*}
  we therefore have by Lemma~\ref{lem:energyestimate} that
  \[
  \liminf_{n \to \infty} \int_D \biggl( \int_D \mu (Q(x,y))^{-\eta}
  \, \mathrm{d} \mu_n \biggr)^{-1} \, \mathrm{d} \mu_n (y)
  \geq (1- \eta) \mu (D)^\eta.
  \]
  Lemma~\ref{lem:frostman2} then implies that $\limsup_n E_n \in
  \mathscr{G}_\mu^\eta$. As this holds for all $\eta < \theta$, we
  conclude that $\limsup_n E_n \in \mathscr{G}_\mu^\theta$.
\end{proof}

\section{Proof of Theorem~\ref{the:randomlimsup}} \label{sec:randomproof}

In this section, we will prove Theorem~\ref{the:randomlimsup}.  We
assume that $\mu$ is a Borel probability measure on $\mathbb{R}^d$ and
consider the random sequence of points $(x_n)_{n=1}^\infty$, where the
points $x_n$ are independent and distributed according to the measure
$\mu$. Recall that we are considering the random set
\[
E_\alpha = \limsup_n B_n
\]
where $B_n = B(x_n,n^{-\alpha})$ and that
\[
s = s(\mu) = \lim_{\varepsilon \to 0} \sup \{\, t : G_\mu (t) \geq t
- \varepsilon \,\}.
\]

\subsection{Definition of the measure $\boldsymbol{\nu}$}

We will use the theory developed so far in this paper. Since $\mu$ is
a probability measure we can choose the point $P$ which defines the
dyadic cubes in such a way that $\mu (R) = 0$.

We will prove that
\begin{equation*} 
  E_\alpha \in \mathscr{G}_\mu^\theta, \qquad \theta = \frac{1}{s
    \alpha}
\end{equation*}
holds almost surely. To do so, we shall pick a set $C$, and first
prove that $E_\alpha \in \mathscr{G}_\nu^\theta$ with $\nu =
\frac{\mu|_C}{\mu (C)}$. Using a limit argument, we will then conclude
that $E_\alpha \in \mathscr{G}_\mu^\theta$ holds almost surely. The
dimension estimate will follow from Theorem~\ref{the:dimension} and
the fact that $E_\alpha \in \mathscr{G}_\nu^\theta$, and is a stronger
estimate than we would have got from Theorem~\ref{the:dimension} using
only that $E_\alpha \in \mathscr{G}_\mu^\theta$.

We will assume that $\mu$ has compact support. This is no restriction,
since when proving that $E_\alpha \in \mathscr{G}_\mu^\theta$ and
$E_\alpha \in \mathscr{G}_\nu^\theta$, we may consider instead the
measure restricted to dyadic cubes.

Let $\varepsilon > 0$. Take $t_0 = \sup \{\, t : G_\mu (t) \geq t - 2
\varepsilon \,\}$ and let $G = \sup \{\, G_\mu (t) : t \geq t_0
\,\}$. Both $t_0$ and $G$ are finite, since $G_\mu$ is a bounded
function.

For each $n$, let $A_n$ be the union of those cubes of $\mathscr{D}_n$
which have $\mu$-measure at most $2^{- n t_0}$, that is
\[
A_n = \bigcup_{\substack{D \in \mathscr{D}_n \\ \mu (D) \leq 2^{-n
      t_0}}} D.
\]

\begin{lemma} \label{lem:t0}
  There is a number $n_0$ such that $\mu (A_n) \leq 2^{- \varepsilon
    n}$ for $n > n_0$.
\end{lemma}

\begin{proof}
  Let $t \geq t_0$ and $\delta_0 > 0$. There is a $\delta (t) \in (0,
  \delta_0)$ such that
  \[
  \limsup_{r \to 0} \frac{\log (N(t + \delta (t), r) - N(t - \delta
    (t), r))}{- \log r} \leq G_\mu (t) + \frac{1}{4} \varepsilon,
  \]
  and hence there is also an $r_0 (t)$ such that
  \[
  \frac{\log (N(t + \delta (t), r) - N(t - \delta (t), r))}{- \log r}
  \leq G_\mu (t) + \frac{1}{2} \varepsilon, \qquad r < r_0 (t).
  \]
  Hence we have
  \[
  N(t + \delta (t), r) - N(t - \delta (t), r) \leq r^{-(G_\mu (t) +
    \frac{1}{2} \varepsilon)}, \qquad r < r_0 (t).
  \]
  By compactness, there are $t_0 \leq t_1 \leq \cdots \leq t_p \leq d
  + 2 \varepsilon$ such that the intervals $B(t_k, \delta (t_k))$, $1
  \leq k \leq p$ cover $[t_0, d + 2 \varepsilon]$.  Let $\delta = \max
  \{\delta (t_1), \ldots \delta (t_p) \} < \delta_0$ and $r_0 = \min
  \{ r_0 (t_1), \ldots, r_0 (t_p) \}$.

  Those cubes of $\mathscr{D}_n$ that have positive measure which is
  at most $2^{-(d + 2 \varepsilon) n}$ are at most $c 2^{dn}$, where $c$
  is a constant, since there are at most $c 2^{dn}$ cubes of
  $\mathscr{D}_n$ with positive measure (since the support is
  compact).

  Take $N$ such that $2^{-N} < r_0$. We then have for $n > N$
  that
  \[
  \mu (A_n) \leq c 2^{dn} 2^{-(d + 2 \varepsilon) n} + \sum_{k = 1}^p
  2^{(G_\mu (t_k) + \frac{1}{2} \varepsilon) n} 2^{-(t_k - \delta) n}
  \leq c 2^{- 2 \varepsilon n} + p 2^{(\delta - \frac{3}{2}
    \varepsilon)n},
  \]
  since $G_\mu (t_k) \leq t_k - 2 \varepsilon$. Since $\delta_0$ is
  arbitrary and $\delta < \delta_0$, we can make $\delta$ as small as
  we like, and hence we can achieve that
  \[
  \mu (A_n) \leq (p+1) 2^{- \frac{3}{2} \varepsilon n}
  \]
  for $n > N$. Hence $\mu (A_n) \leq 2^{- \varepsilon n}$ if $n$ is
  large enough.
\end{proof}

Put
\begin{align*}
  \Theta_n (t) &= \bigcap_{k=n}^\infty \{\, x : \mu (B(x,r)) < r^t
  \text{ for } r = 2^{- k} \,\}, \\ \Theta (t) &= \bigcap_{n=1}^\infty
  \Theta_n (t).
\end{align*}
If $t_1 < \udimh \mu$ then $\mu (\Theta (t_1)) > 0$ and so $\mu
(\Theta_{n_1} (t_1)) > 0$ if $n_1$ is large enough.
We assume from now on that $t_1 < \udimh \mu$, and we choose a number
$m_1$ so large that $\mu (\Theta_{m_1} (t_1)) > 0$.

Let
\[
\mathscr{C}_m = \{\, D \in \mathscr{D}_{[\beta m]} : \mu (\Theta_{m_1} (t_1) \cap
D) \geq \frac{5}{6} \mu (D) \,\}.
\]
(We use $\mathscr{D}_{[\beta m]}$ instead of $\mathscr{D}_m$, since we
are later going to use that the balls $B_m$ which define the
limsup-set $E_\alpha$ have the property that $D \subset B_m$ for some
$D \in \mathscr{D}_{[\beta m]}$.)

By Lebesgue density theorem,
\[
\mu (\Theta_{m_1} (t_1) \setminus \cup \mathscr{C}_m) \to 0, \qquad m
\to \infty,
\]
where we have used the notation
\[
\cup \mathscr{A} = \bigcup_{A \in \mathscr{A}} A.
\]
This notation will appear several times below.

There exists a sequence $(m_k)$ such that the set
\[
\hat{\Theta} = \Theta_{m_1} (t_1) \cap \bigcap_{k=1}^\infty \cup
\mathscr{C}_{m_k}
\]
satisfies $\mu (\hat{\Theta}) > \frac{1}{2} \mu (\Theta_{m_1} (t_1)) >
0$ and
\[
  \sum_{l = k + 1} \mu (\Theta_{m_1} (t_1) \setminus \cup
  \mathscr{C}_{m_k}) < \frac{1}{6} 2^{-[\beta m_k] t_0}.
\]
By construction, the set $\hat{\Theta}$ has the following property. If
$D \in \mathscr{D}_{[\beta m_k]}$ for some $m_k$, then either $D
\subset \hat{\Theta}^\complement$ or
\begin{align} \label{eq:propertyofDelta}
  \mu (\hat{\Theta} \cap D) &\geq \mu (\Theta_{m_1} (t_1) \cap D) -
  \sum_{l = k + 1} \mu (\Theta_{m_1} (s) \setminus \cup
  \mathscr{C}_{m_k}) \nonumber \\ &\geq \frac{5}{6} \mu (D) -
  \frac{1}{6} 2^{-[\beta m_k] t_0}.
\end{align}

Let $\beta > \alpha$. We choose a sparse sequence $(n_k)$, which is a
subsequence of $(m_k)$, as follows. By Lemma~\ref{lem:t0}, there is a
number $n_0$ such that $\mu (A_n) \leq 2^{- \varepsilon n}$ for $n
\geq n_0$. For $k \geq 1$, we may choose inductively a sequence $n_k$
such that $\beta n_1 > n_0$ and so that
\begin{equation} \label{eq:nk}
  \frac{1}{6} 2^{- [\beta n_k] t_0} > \sum_{l =
    k+1}^\infty 2^{- \varepsilon [\beta n_l]}
\end{equation}
holds for all $k \geq 1$. The sequence $n_k$ can in fact be chosen
according to the following lemma.

\begin{lemma} \label{lem:nu}
  The sequence $n_k$ can be chosen so that the following holds.
  If $D \in \mathscr{D}_{[\beta n_k]}$ for some $k$, then have
  \[
  \nu (D) = 0 \text{ or } \nu (D) > 2^{- t_0 [\beta n_k] - 1}.
  \]
\end{lemma}

\begin{proof}
  Let $B_n = A_n^\complement$, and
  \[
  C = \bigcap_{k=1}^\infty B_{[\beta n_k]} \cap \hat{\Theta}.
  \]
  We have
  \begin{align*}
    \mu (C) &\geq \mu (\hat{\Theta}) - \sum_{k=1}^\infty \mu
    (A_{[\beta n_k]}) \geq \mu (\hat{\Theta}) - \sum_{k=1}^\infty 2^{-
      \varepsilon [\beta n_k]} \\ &\geq \mu (\hat{\Theta}) -
    \sum_{k=[\beta n_1]}^\infty 2^{-\varepsilon k} = \mu
    (\hat{\Theta}) - \frac{2^{-\varepsilon [\beta n_1]}}{1 -
      2^{-\varepsilon}}.
  \end{align*}
  Hence, if $t_1 < \udimh \mu$, then we can ensure that $\mu (C) >
  \frac{1}{2} \mu (\hat{\Theta}) > 0$ by choosing $n_1$ sufficiently
  large. We assume that $\mu (C) > \frac{1}{2} \mu (\hat{\Theta})$
  holds.

  We now put $\nu = \frac{\mu|_C}{\mu (C)}$. Clearly, $\nu$ is
  absolutely continuous with respect to $\mu$, with bounded
  density. Moreover, by the choice of the sequence $n_k$, if $D \in
  \mathscr{D}_{[\beta n_k]}$ for some $k$, then
  \[
  \nu (D) \leq \frac{\mu (D)}{\mu (C)} \leq \frac{2}{\mu(\hat{\Theta})}
  \mu (D),
  \]
  and if $\mu (D) > 2^{- t_0 [\beta n_k]}$ then, either $D \subset
  \hat{\Theta}^\complement \cup A_{[\beta n_l]}$ for some $l < k$ and $\nu (D) =
  0$, or
  \begin{align*}
    \nu (D) &\geq \mu (C \cap D) \geq \mu (\hat{\Theta} \cap D) - \mu
    \biggl( \bigcup_{l = k+1}^\infty A_{[\beta n_l]} \biggr) \\ &\geq
    \frac{5}{6} \mu (D) - \frac{1}{6} 2^{-[\beta n_k]t_0} -
    \sum_{l=k+1}^\infty 2^{-\varepsilon [\beta n_l]} \\ &\geq
    \frac{5}{6} \mu (D) - \frac{1}{6} 2^{-[\beta n_k]t_0} -
    \frac{1}{6} 2^{-[\beta n_k] t_0} \geq \frac{1}{2} \mu (D),
  \end{align*}
  holds by \eqref{eq:propertyofDelta} and \eqref{eq:nk}.

  Finally, if $D \in \mathscr{D}_{[\beta n_k]}$ is such that $\mu (D)
  \leq 2^{- t_0 [\beta n_k]}$, then $D \subset A_{[\beta n_k]}$ and so
  $\nu (D) = 0$.
\end{proof}

\subsection{Almost surely $\boldsymbol{E_\alpha \in \mathscr{G}_\nu^\theta}$}

In this section, we are going to prove that $E_\alpha \in
\mathscr{G}_\nu^\theta$ holds almost surely, using to use
Lemma~\ref{lem:frostman}. We will therefore check that the conditions
of Lemma~\ref{lem:frostman} holds almost surely.

Take $\beta > \alpha$ such that $\frac{1}{\beta} \leq t_0$. We define
the random probability measures
\[
\nu_n = 2^{-(n-1)} \sum_{k=2^{n-1}+1}^{2^n} \frac{\nu|_{D_k}}{\nu
  (D_k)},
\]
where $D_k \in \mathscr{D}_{[ \beta n ]}$ are chosen so that $x_k \in
D_k$. If $\nu (D_k) = 0$, then $\frac{\nu|_{D_k}}{\nu (D_k)}$ should
be interpreted as $0$. It is then clear that almost surely, $\nu_n$
converge weakly to $\nu$, and that
\[
\supp \nu_n \subset \bigcup_{k=2^{n-1}+1}^{2^n} B (x_k, k^{-\alpha})
\]
if $n$ is large.

We want to estimate the expectation of $\iint \nu(Q(x,y))^{-\theta} \,
\mathrm{d} \nu_n \mathrm{d} \nu_n$. Let $\E$ denote expectation and
split the expectation into two parts,
\[
\E \iint \nu (Q(x,y))^{-\theta} \, \mathrm{d} \nu_n (x) \mathrm{d}
\nu_n (y) = E_1 + E_2,
\]
where
\begin{align*}
  E_1 &= 2^{-2n-2} \sum_{k \neq l} \E \int_{D_k} \int_{D_l} \nu
  (Q(x,y))^{-\theta} \, \frac{\mathrm{d} \nu (x)}{\nu (D_l)}
  \frac{\mathrm{d} \nu (y)}{\nu (D_k)}, \\ E_2 &= 2^{-2n-2}
  \sum_{k=2^{n-1}+1}^{2^n} \E \int_{D_k} \int_{D_k} \nu
  (Q(x,y))^{-\theta} \, \frac{\mathrm{d} \nu (x)}{\nu (D_k)}
  \frac{\mathrm{d} \nu (y)}{\nu (D_k)}.
\end{align*}

We first consider $E_1$.

\begin{lemma} \label{lem:E1}
  For $k \neq l$ we have
  \[
  \E \int_{D_k} \int_{D_l} \nu (Q(x,y))^{-\theta} \, \frac{\mathrm{d}
    \nu (x)}{\nu (D_l)} \frac{\mathrm{d} \nu (y)}{\nu (D_k)} = \iint
  \nu(Q(x,y))^{-\theta} \, \mathrm{d} \nu (x) \mathrm{d} \nu (y),
  \]
  and
  \[
  E_1 = (1 - 2^{-(n-1)}) \iint \nu(Q(x,y))^{-\theta} \, \mathrm{d} \nu
  (x) \mathrm{d} \nu (y).
  \]
\end{lemma}

\begin{proof}
  We have
  \begin{multline*}
    \E \int_{D_k} \int_{D_l} \nu (Q(x,y))^{-\theta} \,
    \frac{\mathrm{d} \nu (x)}{\nu (D_l)} \frac{\mathrm{d} \nu
      (y)}{\nu (D_k)} \\ = \iint \int_{D_k} \int_{D_l} \nu
    (Q(x,y))^{-\theta} \, \frac{\mathrm{d} \nu (x)}{\nu (D_l)}
    \frac{\mathrm{d} \nu (y)}{\nu (D_k)} \, \mathrm{d} \nu (x_k)
    \mathrm{d} \nu (x_l).
  \end{multline*}
  If $x_k$ and $x_l$ are such that $D_k \neq D_l$, then $Q(x,y)$ is
  constant for $(x,y) \in D_k \times D_l$ and in fact $Q(x,y) =
  Q(x_k,x_l)$. This is not the case if $D_k = D_l$. But if $D \in
  \mathscr{D}_{[ \beta n ]}$, then
  \[
  \int_{D_k} \int_{D_l} \nu (Q(x,y))^{-\theta} \, \frac{\mathrm{d}
    \nu (x)}{\nu (D_l)} \frac{\mathrm{d} \nu (y)}{\nu (D_k)} =
  \int_D \int_D \nu (Q(x,y))^{-\theta} \, \frac{\mathrm{d} \nu
    (x)}{\nu (D)} \frac{\mathrm{d} \nu (y)}{\nu (D)}
  \]
  whenever $x_k$ and $x_l$ are such that $D_k = D_l = D$, that is for
  $x_k, x_l \in D$.

  Hence we have
  \begin{align*}
    \iint & \int_{D_k} \int_{D_l} \nu (Q(x,y))^{-\theta} \,
    \frac{\mathrm{d} \nu_n (x)}{\nu (D_l)} \frac{\mathrm{d} \nu_n
      (y)}{\nu (D_k)} \, \mathrm{d} \nu (x_k) \mathrm{d} \nu (x_l)
    \\ = & \iint\limits_{D_k \neq D_l} \nu (Q(x_k,x_l))^{-\theta} \,
    \mathrm{d}\nu (x_k) \mathrm{d}\nu (x_l) \\ & + \sum_{D \in
      \mathscr{D}_{[\beta n]}} \iint_{D\times D} \iint_{D\times D} \nu
    (Q(x,y))^{-\theta} \, \frac{\mathrm{d} \nu (x)}{\nu (D)}
    \frac{\mathrm{d} \nu (y)}{\nu (D)} \mathrm{d} \nu (x_k) \mathrm{d}
    \nu (x_l) \\ = & \iint\limits_{D_k \neq D_l} \nu
    (Q(x_k,x_l))^{-\theta} \, \mathrm{d}\nu (x_k) \mathrm{d}\nu (x_l)
    \\ & + \sum_{D \in \mathscr{D}_{[\beta n]}} \iint_{D\times D} \nu
    (Q(x_k,x_l))^{-\theta} \, \mathrm{d} \nu (x_k) \mathrm{d} \nu
    (x_l) \\ = & \iint \nu(Q(x,y))^{-\theta} \, \mathrm{d} \nu (x)
    \mathrm{d} \nu (y),
  \end{align*}
  which is the first equality of the lemma. The second equality of the
  lemma follows immediately from the first equality, since the first
  equality says that all terms in the sum defining $E_1$ are equal to
  $\iint \nu(Q(x,y))^{-\theta} \, \mathrm{d} \nu (x) \mathrm{d} \nu
  (y)$.
\end{proof}

We now consider $E_2$. Recall that $D_n (x)$ denotes the unique
element of $\mathscr{D}_n$ that contains $x$.

\begin{lemma} \label{lem:E2}
  We have
  \[
  \E \int_{D_k} \int_{D_k} \nu(Q(x,y))^{-\theta} \, \frac{\mathrm{d}
    \nu (x)}{\nu (D_k)} \frac{\mathrm{d} \nu (y)}{\nu (D_k)} \leq
  \frac{1}{1-\theta} \int \nu (D_{[ \beta n ]} (x))^{-\theta} \,
  \mathrm{d} \nu (x)
  \]
  and
  \[
  E_2 \leq \frac{2^{-(n-1)}}{1-\theta} \int \nu (D_{[\beta n]}
  (x))^{-\theta} \, \mathrm{d} \nu (x).
  \]
\end{lemma}

\begin{proof}
  By Lemma~\ref{lem:energyestimate}
  \[
  \int_{D_k} \int_{D_k} \nu(Q(x,y))^{-\theta} \, \frac{\mathrm{d} \nu
    (x)}{\nu (D_k)} \frac{\mathrm{d} \nu (y)}{\nu (D_k)} \leq
  \frac{\nu (D_k)^{-\theta}}{1-\theta}.
  \]
  Hence
  \[
  \E \int_{D_k} \int_{D_k} \nu(Q(x,y))^{-\theta} \, \frac{\mathrm{d}
    \nu (x)}{\nu (D_k)} \frac{\mathrm{d} \nu (y)}{\nu (D_k)} \leq
  \frac{1}{1-\theta} \int \nu (D_{[\beta n]} (x))^{-\theta} \,
  \mathrm{d} \nu (x),
  \]
  which is the first estimate in the lemma. The second estimate of the
  lemma follows from the first estimate and the definition of $E_2$.
\end{proof}

From Lemma~\ref{lem:E1} and \ref{lem:E2} we get the following
corollary.

\begin{corollary} \label{cor:expectation}
  \begin{multline*}
    \E \iint \nu (Q(x,y))^{-\theta} \, \mathrm{d} \nu_n (x) \mathrm{d}
    \nu_n (y) \\ \leq \iint \nu (Q(x,y))^{-\theta} \, \mathrm{d} \nu
    (x) \mathrm{d} \nu (y) + \frac{2^{-(n-1)}}{1-\theta} \int \nu
    (D_{[ \beta n ]} (x))^{-\theta} \, \mathrm{d} \nu (x).
  \end{multline*}
\end{corollary}

We need to estimate $\int \nu (D_{[ \beta n ]} (x))^{-\theta} \,
\mathrm{d} \nu(x)$. By Lemma~\ref{lem:nu}, we have that either $\nu
(D_{[\beta n_k]} (x)) = 0$ or $\nu (D_{[\beta n_k]} (x)) > 2^{-t_0
  [\beta n_k] - 1}$. We therefore have
\[
\int \nu (D_{[\beta n_k]} (x))^{-\theta} \, \mathrm{d} \nu(x) <
2^{(t_0 \beta n_k + 1) \theta} \leq 2^{t_0 \beta n_k \theta + 1}.
\]
Hence, by Corollary~\ref{cor:expectation}, if $t_0 \beta \theta \leq
1$, then the expectations
\[
\E \iint \nu (Q(x,y))^{-\theta} \, \mathrm{d} \nu_{n_k} (x) \mathrm{d}
\nu_{n_k} (y)
\]
are uniformly bounded. Take $\theta = \frac{1}{t_0 \beta}$. Then
$\theta \leq 1$ since $\frac{1}{\beta} \leq t_0$.

Almost surely, there is then a sequence $m_k$, which is a subsequence
of $n_k$, and such that
\[
\iint \nu (Q(x,y))^{-\theta} \, \mathrm{d} \nu_{m_k} (x) \mathrm{d}
  \nu_{m_k} (y)
\]
is uniformly bounded and we assume that $m_k$ is such a sequence.

Lemma~\ref{lem:frostman} then implies that $\limsup_k E_{m_k} \in
\mathscr{G}_\nu^\theta$ with $\theta = (t_0 \beta)^{-1}$. Hence we
have almost surely that $E_\alpha \in \mathscr{G}_\nu^\theta$ with
$\theta = (t_0 \beta)^{-1}$.  Since $t_0$ can be chosen arbitrary
close to $s (\mu)$ and $\beta$ can be taken arbitrary close to
$\alpha$, this proves that almost surely $E_\alpha \in
\mathscr{G}_\nu^\theta$ with $\theta = (s \alpha)^{-1}$. By taking
$t_1$ close to $\udimh \mu$, the first dimension estimate of
Theorem~\ref{the:dimension} implies that
\[
\dimh E_\alpha \geq \frac{1}{\alpha} \frac{\udimh \mu}{s}
\]
almost surely.

We have shown that almost surely $E_\alpha \in \mathscr{G}_\nu^\theta$
with $\theta = (t_0 \beta)^{-1}$. This means, by
Definition~\ref{def:class}, that for $\eta < \theta$, we have
\[
\mathscr{M}_\nu^\eta (E \cap D) = \mathscr{M}_\nu^\eta (D)
\]
whenever $D$ is a dyadic cube. Since $\nu = \frac{\mu|_C}{\mu(C)}$, we
have
\[
\mu (C)^{-\eta} \mathscr{M}_\mu^\eta (E \cap D) \geq \mathscr{M}_\nu^\eta
(E \cap D) = \mathscr{M}_\nu^\eta (D) = \nu (D)^\eta = \mu(C \cap D)^\eta.
\]
We can make $\mu(C)$ as close to $1$ as we like (if we let $t_1 <
\ldimh \mu$), and therefore we almost surely have
\[
\mathscr{M}_\mu^\eta (E \cap D) \geq \mu(D)^\eta =
\mathscr{M}_\mu^\eta (D).
\]
Hence $E_\alpha \in \mathscr{G}_\mu^\theta$ almost surely, with
$\theta = (s \alpha)^{-1}$.

\section{Proofs of result on dynamical Diophantine approximation} \label{sec:shrinkingproofs}

In this section we will prove Theorem~\ref{the:shrinking} and
Corollaries~\ref{cor:quadratic} and \ref{cor:piecewise}.

We start by proving the corollaries, which are consequences of
Theorem~\ref{the:shrinking}.

\begin{proof}[Proof of Corollary~\ref{cor:quadratic}]
  From Young \cite{young} we know that when $a \in \Delta$, there is a
  $T_a$-invariant measure $\mu_a$ with properties described in the
  following.

  The support of $\mu_a$ is $[T_a^2 (0), T_a (0)]$, $\mu_a$ is
  absolutely continuous with respect to Lebesgue measure and the
  density of $\mu_a$ is bounded away from $0$ on the support
  \cite[Theorem~2]{young}. Hence, we have $s=1$, where $s$ is defined
  as in Theorem~\ref{the:shrinking}.

  The density $\rho$ of $\mu_a$ can be written as $\rho = \rho_1 +
  \rho_2$, where 
  \[
  0 \leq \rho_2 (x) \leq C \sum_{k=1}^\infty \frac{1.9^{-k}}{\sqrt{|x
      - T_a^k (0)|}}
  \]
  and $\rho_1$ is bounded \cite[Theorem~1]{young}. It follows that we
  may take $t_1 = \frac{1}{2}$ in \eqref{eq:measuredecay}.

  Correlations decay exponentially in the sense that \eqref{eq:decay}
  holds with $p(n) = C \tau^n$ for some $C$ and $\tau \in (0,1)$. This
  is apparent from the proof of the main theorem in \cite{young}. (The
  actual theorem contains a somewhat weaker statement.)

  Hence all assumptions of Theorem~\ref{the:shrinking} are satisfied,
  and we conclude that there is a set $A$ of full $\mu_a$ measure in
  $[1-a,1]$ such that
  \[
  \dimh \bigcap_n E_\alpha (x_n) = \frac{1}{\alpha}
  \]
  holds whenever $x_1, x_2, \ldots$ are elements of $A$. Since $\mu_a$
  is equivalent to Lebesgue measure on $[1-a,1]$, this proves the theorem.
\end{proof}

\begin{proof}[Proof of Corollary~\ref{cor:piecewise}]
  By Corollary~3 of \cite{perssonrams}, assumption
  \eqref{eq:measuredecay} is satisfied with
  \[
  t_1 = \limsup_{m\to \infty} \inf \frac{S_m \phi - m P(\phi)}{- \log
    |(T^m)'|} > 0.
  \]
  By Theorem~\ref{the:shrinking}, there is a set $A$ of full
  $\mu_\phi$ measure such that
  \[
  \frac{1}{\alpha} \frac{\dimh \mu_\phi}{s} = \frac{1}{\alpha}
  \frac{\udimh \mu_\phi}{s} \leq \dimh \bigcap_k E_\alpha (x_k) \leq
  \frac{1}{\alpha}
  \]
  holds whenever $x_1, x_2, \ldots$ are elements of
  $A$. Assumption~\eqref{eq:dim-coarse} implies that $\dimh \mu_\phi =
  s$.
\end{proof}

\subsection{Proof of Theorem~\ref{the:shrinking}}

Let us now turn to the proof of Theorem~\ref{the:shrinking}. By a
change of variables, we may assume that $T$ is a map of the unit
interval $[0,1]$. Since $\mu$ is not atomic and the space is
one-dimensional, $\mu (R) = 0$ holds.

We shall proceed as in the proof of
Theorem~\ref{the:randomlimsup}. The setting is similar since
\[
E_\alpha (x) = \limsup_n B(T^n (x), n^{-\alpha}),
\]
but some further complications occur since the points $T^n (x)$ are
not independent. However, the assumption on the decay of correlations
gives us enough asymptotic independence to carry out the proof.

As in the proof of Theorem~\ref{the:randomlimsup}, we put
\begin{align*}
  \Theta_n (t) &= \bigcap_{k=n}^\infty \{\, x : \mu (B(x,r)) < r^t
  \text{ for } r = 2^{- k} \,\}, \\ \Theta (t) &= \bigcap_{n=1}^\infty
  \Theta_n (t),
\end{align*}
and take $t_2 < \udimh \mu$, and $m_1$ so large that $\mu
(\Theta_{m_1} (t_2)) > 0$.

We let as before
\[
\mathscr{C}_m = \{\, D \in \mathscr{D}_{[\beta m]} : \mu (\Theta_{m_1}
(t_2) \cap D) \geq \frac{5}{6} \mu (D) \,\},
\]
and there exists a sequence $(m_k)$ such that the set
\[
\hat{\Theta} = \Theta_{m_1} (t_2) \cap \bigcap_{k=1}^\infty \cup
\mathscr{C}_{m_k}.
\]
satisfies $\mu (\hat{\Theta}) > \frac{1}{2} \mu (\Theta_{m_1} (t_2)) >
0$. By construction, the set $\hat{\Theta}$ has the following
property. If $D \in \mathscr{D}_{[\beta m_k]}$ for some $m_k$, then
either $D \subset \hat{\Theta}^\complement$ or
\begin{equation} \label{eq:propertyofDelta2}
  \mu (\hat{\Theta} \cap D) \geq \frac{5}{6} \mu (D) - \frac{1}{6}
  2^{-[\beta m_k]t_0}.
\end{equation}

Let $\varepsilon > 0$ and $t_0 = \sup \{\, t: G_\mu (t) \geq t - 2
\varepsilon \,\}$. Take $\beta > \alpha$ such that $\frac{1}{\beta}
\leq t_0$. Let $G = \sup \{\, G_\mu (t) : t \geq t_0 \,\}$. As in the
proof of Theorem~\ref{the:randomlimsup}, we define a sparse sequence
$n_k$ as follows.

For each $n$, let $A_n$ be the union of those cubes of $\mathscr{D}_n$
which have $\mu$-measure at most $2^{-n t_0}$. By Lemma~\ref{lem:t0},
there is a number $n_0$ such that $\mu (A_n) \leq 2^{-\varepsilon n}$
for $n > n_0$.

We choose the sequence $n_k$ so that $\{n_k\} \subset \{m_k\}$ and
\begin{equation} \label{eq:nk2}
  \frac{1}{4} 2^{- [\beta n_k] t_0} > \sum_{l = k+1}^\infty 2^{-
    \varepsilon [\beta n_l]}
\end{equation}
holds for all $k \geq 1$ and $\beta n_1 > n_0$. Let $B_n = A_n^\complement$ and
\[
C = \bigcap_{k=1}^\infty B_{[\beta n_k]} \cap \hat{\Theta}.
\]
We have
\begin{align*}
  \mu (C) &\geq \mu (\hat{\Theta}) - \sum_{k=1}^\infty \mu (A_{[\beta
      n_k]}) \geq \mu (\hat{\Theta}) - \sum_{k=1}^\infty 2^{-
    \varepsilon [\beta n_k]} \\ &\geq \mu (\hat{\Theta}) -
  \sum_{k=[\beta n_1]}^\infty 2^{-\varepsilon k} = \mu (\hat{\Theta})
  - \frac{2^{-\varepsilon [\beta n_1]}}{1 - 2^{-\varepsilon}}.
\end{align*}
Hence, if $t_1 < \ldimh \mu$, by choosing $n_1$ sufficiently large, we
can ensure that $\mu (C) > \frac{1}{2}$, and we can make $\mu (C)$ as
close to $1$ as we please. Similarly, if $t_1 < \udimh \mu$, then we
can ensure that $\mu (C) > \frac{1}{2} \mu (\hat{\Theta}) > 0$ by
choosing $n_1$ sufficiently large. We assume that $\mu (C) >
\frac{1}{2} \mu (\hat{\Theta})$ holds.

Let $f$ be the indicator function of the set $C$. Since $\frac{1}{2}
\mu (\hat{\Theta}) < \int f \, \mathrm{d} \mu < 1$, we have by
Birkhoff's ergodic theorem that for almost every $x$, there exists a
number $N(x)$ such that
\begin{equation} \label{eq:frequency}
  2^{-(n-1)} \sum_{k = 2^{n-1} + 1}^{2^n} f (T^k x) > \frac{1}{2} \mu
  (\hat{\Theta})
\end{equation}
for all $n > N (x)$.

Let
\[
\mu_n = 2^{-(n-1)} \sum_{k = 2^{n-1} + 1}^{2^n} \frac{\mu |_{D_{[\beta
        n]} (T^k x)}}{\mu (D_{[\beta n]} (T^k x))},
\]
where $\frac{0}{0}$ should be interpreted as $0$. The measure $\mu_n$
depends on $x$, but we suppress this dependence from the notation. If
$n$ is large we have
\[
\supp \mu_n \subset \bigcup_{k = 2^{n-1}+1}^{2^n} B(x,k^{-\alpha}),
\]
since $\beta > \alpha$. Moreover, for almost all $x$, by Birkhoff's
theorem, $\mu_n \to \mu$ weakly as $n \to \infty$.

Put $\nu_k = c_k(x) \mu_{n_k}|_C$, where $c_k (x)$ is a constant
chosen so that $\nu_k$ is a probability measure, that is
\[
c_k (x) = \frac{1}{\mu_{n_k} (C)}.
\]
This makes $c_k (x)$ well-defined for almost all $x$ if $k$ is larger
than some number $K(x)$. In fact, we have by \eqref{eq:frequency} that
\[
c_k (x) \in \biggl[1,\frac{2}{\mu(\hat{\Theta})} \biggr), \qquad
  \text{if } n_k > N(x).
\]

By Birkhoff's ergodic theorem, we have that $\nu_k \to \nu$ weakly as
$k \to \infty$, where $\nu = \frac{\mu|_C}{\mu (C)}$.
  
We consider the $(\nu, \theta)$-energies 
\[
I_\nu^\theta (\nu_k) = \iint \nu (Q)^{-\theta} \, \mathrm{d} \nu_k
\mathrm{d} \nu_k.
\]
Our aim is to show that, almost surely, there is a sequence $k_l$
along which $I_\mu^\theta (\nu_{k_l})$ are uniformly bounded.

Given an integer $m$ we define an approximation of the function $Q$,
introduced in Section~\ref{sec:inhomogeneous}. Let $Q_m \colon
\mathbb{R}^d \times \mathbb{R}^d \to \mathscr{D}$ be defined by
\[
Q_m (x,y) = D, \qquad x \neq y,
\]
where $D \in \mathscr{D}_n$ is chosen such that $x,y \in D$ and $n$ is
the largest integer with this property such that $n \leq m$. We then
have that
\[
Q (x,y) = Q_m (x,y)
\]
if and only if $Q (x,y) = D$ with $D \in \mathscr{D}_n$ and $n \leq
m$.
  
By the definition of $\nu_k$ we can write
\[
  I_\nu^\theta (\nu_k) = \frac{c_k (x)^2} {4^{n_k-1}} \sum_{i,j \in J_k}
  \int_{D_i \cap C} \int_{D_j \cap C} \nu (Q)^{-\theta} \,
  \frac{\mathrm{d} \mu}{\mu (D_i)} \frac{\mathrm{d} \mu}{\mu (D_j)},
\]
where $D_i = D_{[\beta n_k]} (T^i x)$ and $J_k = \{2^{n_k-1} +1,
\ldots, 2^{n_k} \}$. Assuming that $c_k (x) \leq 2$, which holds
almost surely if $k$ is large enough, we therefore have
\begin{equation} \label{eq:nuenergy}
  I_\nu^\theta (\nu_k) = 4^{2-n_k} \sum_{i,j \in J_k}
  \int_{D_i \cap C} \int_{D_j \cap C} \nu (Q)^{-\theta} \,
  \frac{\mathrm{d} \mu}{\mu (D_i)} \frac{\mathrm{d} \mu}{\mu (D_j)}.
\end{equation}

If $\mu (D_i) \leq 2^{-[\beta n_k] t_0}$, then $D_i \subset A_{[\beta
    n_k]}$, so that $D_i \subset C^\complement$ and $\nu (D_i) = 0$. Otherwise,
if $\mu (D_i) > 2^{-[\beta n_k] t_0}$ then, either $D_i \subset
\hat{\Theta}^\complement \cup A_{[\beta n_l]}$ for some $l < k$ and $\nu (D_i) =
0$, or
\begin{align*}
  \nu (D) &\geq \mu (C \cap D) \geq \mu (\hat{\Theta} \cap D) - \mu
  \biggl( \bigcup_{l = k+1}^\infty A_{[\beta n_l]} \biggr) \\ &\geq
  \mu (\hat{\Theta} \cap D) - \sum_{l=k+1}^\infty 2^{-\varepsilon
    [\beta n_l]} \\ &\geq \frac{5}{6} \mu (D) - \frac{1}{6} 2^{-[\beta
      n_k] t_0} - \frac{1}{6} 2^{-[\beta n_k] t_0} \geq \frac{1}{2}
  \mu (D),
\end{align*}
holds by \eqref{eq:nk2} and \eqref{eq:propertyofDelta2}. In the other
direction we have
\[
\nu (D_i) = \frac{\mu (C \cap D_i)}{\mu (C)} \leq \frac{2}{\mu
  (\hat{\Theta})} \mu (D_i).
\]
We conclude that either $\nu (D_i) = 0$ or
\begin{equation} \label{eq:mu-nu-ratio}
  \frac{1}{2} \leq \frac{\nu (D_i)}{\mu (D_i)} \leq \frac{2}{\mu
    (\hat{\Theta})}.
\end{equation}
We also have
\begin{equation} \label{eq:nuestimate}
  D_i \cap C = \emptyset \qquad \text{or} \qquad \nu (D_i) \geq
  2^{-[\beta n_k] t_0 - 1}.
\end{equation}

Combining \eqref{eq:nuenergy} and \eqref{eq:mu-nu-ratio}, we obtain
\begin{align*}
  I_\nu^\theta (\nu_k) &\leq \frac{4^{3-n_k}}{\mu (\hat{\Theta})^2}
  \sum_{i,j \in J_k} \int_{D_i \cap C} \int_{D_j \cap C} \nu
  (Q)^{-\theta} \, \frac{\mathrm{d} \mu}{\nu (D_i)} \frac{\mathrm{d}
    \mu}{\nu (D_j)} \\ &= 4^{3-n_k} \frac{\mu(C)^2}{\mu
    (\hat{\Theta})^2} \sum_{i,j \in J_k} \int_{D_i} \int_{D_j} \nu
  (Q)^{-\theta} \, \frac{\mathrm{d} \nu}{\nu (D_i)} \frac{\mathrm{d}
    \nu}{\nu (D_j)} \\ &\leq 4^{3-n_k} \sum_{i,j \in J_k} \int_{D_i}
  \int_{D_j} \nu (Q)^{-\theta} \, \frac{\mathrm{d} \nu}{\nu (D_i)}
  \frac{\mathrm{d} \nu}{\nu (D_j)}.
\end{align*}

If either $D_i \cap C = \emptyset$ or $D_j \cap C = \emptyset$, then
\[
\int_{D_i} \int_{D_j} \nu (Q)^{-\theta} \, \frac{\mathrm{d} \nu}{\nu
  (D_i)} \frac{\mathrm{d} \nu}{\nu (D_j)} = 0.
\]

Suppose now that both $D_i \cap C \neq \emptyset$ and $D_j \cap C \neq
\emptyset$. We consider three cases.
  
Case 1: If $i = j$ then $D_i = D_j$ and
\begin{align*}
  \int_{D_i} \int_{D_j} \nu (Q)^{-\theta} \, \frac{\mathrm{d} \nu}{\nu
    (D_i)} \frac{\mathrm{d} \nu}{\nu (D_j)} &= \int_{D_i} \int_{D_i}
  \nu (Q)^{-\theta} \, \frac{\mathrm{d} \nu}{\nu (D_i)}
  \frac{\mathrm{d} \nu}{\nu (D_i)} \\ &\leq \frac{1}{1 - \theta} \nu
  (D_i)^{-\theta} \\ &= \frac{1}{1-\theta} \nu (Q_{[\beta n_k]} (T^i
  x, T^j x))^{-\theta},
\end{align*}
by Lemma~\ref{lem:energyestimate}.

Case 2: If $i \neq j$ and $D_i \neq D_j$, then $Q$ is constant and
equal to $Q(T^i x, T^j x)$ on $D_i \times D_j$ and we get
\[
\int_{D_i} \int_{D_j} \nu (Q)^{-\theta} \, \frac{\mathrm{d} \nu}{\nu
  (D_i)} \frac{\mathrm{d} \nu}{\nu (D_j)} = \nu (Q (T^i x, T^j
x))^{-\theta}.
\]

Case 3: If $i \neq j$ and $D_i = D_j$, then as in the case $i = j$,
we get by Lemma~\ref{lem:energyestimate} that
\begin{align*}
  \int_{D_i} \int_{D_j} \nu (Q)^{-\theta} \, \frac{\mathrm{d} \nu}{\nu
    (D_i)} \frac{\mathrm{d} \nu}{\nu (D_j)} & \leq \frac{1}{1-\theta}
  \nu (D_i)^{-\theta} \\ &= \frac{1}{1-\theta} \nu (Q_{[\beta n_k]}
  (T^i x, T^j x))^{-\theta}.
\end{align*}
  
Taken together, these estimates show that
\begin{multline*}
  \int_{D_i} \int_{D_j} \nu (Q)^{-\theta} \, \frac{\mathrm{d} \nu}{\nu
    (D_i)} \frac{\mathrm{d} \nu}{\nu (D_j)} \\ \leq
  \left\{ \begin{array}{ll} 0 & \text{if } \nu (D_i) \nu (D_j) = 0,
    \\ \frac{1}{1-\theta} \nu (Q_{[\beta n_k]} (T^i x, T^j
    x))^{-\theta} & \text{otherwise}. \end{array} \right.
\end{multline*}
Let
\[
F_k (x,y) = \min\{\nu (Q_{[\beta n_k]} (x,
y))^{-\theta}, 2^{\theta + \theta t_0 \beta n_k} \}.
\]
By \eqref{eq:nuestimate} we have that
\[
\nu (Q_{[\beta n_k]} (x, y))^{-\theta} > 2^{\theta + \theta t_0 \beta
  n_k} \geq 2^{\theta + \theta t_0 [\beta n_k]}
\]
only if $\nu (Q_{[\beta n_k]} (x, y)) = 0$.
We therefore have
\[
(1- \theta) I_\nu^\theta (\nu_k) \leq 4^{3-n_k} \sum_{i,j \in J} F_k
(T^i x, T^j x).
\]

This sum is split into two parts according to
\begin{align*}
  (1 - \theta) I_\nu^\theta (\nu_k) &= I_1 (k) + I_2 (k), \\ I_1 (k)
  &= 4^{3 - n_k} \cdot 2 \sum_{\substack{i,j \in J \\i>j}} F_k (T^i x, T^j x),
  \\ I_2 (k) &= 4^{3 - n_k} \sum_{i \in J} F_k (T^i x, T^i x).
\end{align*}

We will estimate the expected value of $I_1 (k)$.
We have
\begin{align*}
  \E I_1 (k) & \leq 4^{4-n_k} \sum_{\substack{i,j \in J \\i>j}} \int F_k
  (T^i x, T^j x) \, \mathrm{d} \mu (x) \\ &= 4^{4-n_k}
  \sum_{\substack{i,j \in J \\ i>j}} \int F_k (T^{i-j} x, x) \,
  \mathrm{d} \mu (x).
\end{align*}
We will use the decay of correlations to estimate the integrals
above. The following lemma which is a variation of Lemma~3 of Persson
and Rams \cite{perssonrams} will be used.

\begin{lemma} \label{lem:perssonrams}
  Suppose $F \colon [0,1]^2 \to \mathbb{R}$ is a piecewise continuous
  and non-negative function, and that $V$ and $M$ are two constants
  such that, for each fixed $x$, the function $f \colon y \mapsto
  F(x,y)$ satisfies $\var f \leq V$ and $\int f \, \mathrm{d}\mu \leq
  M$. Then
  \[
  \int F(T^n x, x) \, \mathrm{d}\mu (x) \leq M + (V + M) p
  (n).
  \]
\end{lemma}

We will use Lemma~\ref{lem:perssonrams} with $F = F_k$.  Clearly,
$F_k$ is piecewise constant and hence piecewise continuous. Moreover,
$F_k$ is bounded by $2^{\theta + \theta t_0 [\beta n_k]}$ by
\eqref{eq:nuestimate}.

Let $x$ be fixed and consider the function $f_x \colon y \mapsto F_k
(x,y)$. The function $f_x$ is increasing on $[0,x]$ and decreasing on
$[x,1]$. Hence, $f_x$ is of bounded variation, and
\[
\var f_x \leq 2^{1 + \theta + \theta t_0 [\beta n_k]}.
\]
By \eqref{eq:measuredecay}, we have
\[
\int f_x \, \mathrm{d}\mu \leq \sum_{k=0}^{[\beta n]} \mu (D_k
(x))^{1-\theta} \leq M := \sum_{k=0}^\infty c_1^{1-\theta} 2^{-
  (1-\theta) t_1 k} < \infty.
\]
It therefore follows by Lemma~\ref{lem:perssonrams} that
\[
\int F_k (T^{i-j} x, x) \, \mathrm{d} \mu (x) \leq M + (V(k) + M) p
(i-j).
\]
Consequently, using that $p$ is summable, we obtain
\begin{align*}
  \E I_1 (k) &\leq 4^{4-n_k} \sum_{i=2^{n_k-1}+2}^{2^{n_k}}
  \sum_{j=2^{n_k-1}+1}^{i-1} \bigl( M + (V(k) + M) p (i-j) \bigr)
  \\ &= 4^{4-n_k} \sum_{i=2^{n_k-1}}^{2^{n_k}} \bigl( 2^{n_k} M +
  (V(k) + M) \sum_{j=2^{n_k-1}+1}^{i-1} p (i-j) \bigr) \\ &\leq
  4^{4-n_k} \sum_{i=2^{n_k-1}}^{2^{n_k}} \bigl( 2^{n_k} M + (V(k) + M)
  \sum_{j=0}^\infty p (j) \bigr) \\ &\leq 16 M + 16 \cdot 2^{-n_k}
  (V(k) + M) \sum_{i=1}^\infty p(i) \\ & \leq 16 M + c_0 2^{(\beta
    \theta t_0 -1) n_k},
\end{align*} 
where $c_0$ is a constant.
Hence, $\E I_1 (k)$ is uniformly bounded in $k$, provided $\theta \leq
\frac{1}{\beta t_0}$.

To analyse $I_2 (k)$, we write
\begin{align*}
  I_2 (k) &= 4^{3-n_k} \sum_{i \in J} F_k (T^i x, T^i x) \leq
  4^{3-n_k} \sum_{i \in J} 2^{\theta + \theta t_0 \beta n_k} \\ &\leq
  4^{3-n_k} 2^{n_k} 2^{\theta + \theta \beta t_0 n_k} \leq c_1
  2^{(\theta \beta t_0 - 1) n_k}.
\end{align*}
Hence, $I_2 (k)$ is bounded by $c_1$ if $\theta \leq \frac{1}{\beta
  t_0}$.

Take $\theta < \frac{1}{\beta t_0} \leq 1$. Then $I_2 (k)$ is bounded
by $c_1$ for all $k$ and the expectation of $I_1 (k)$ is bounded by
$16M + c_0$ for all $k$. We can conclude that for almost all $x$, that
is almost surely, there exists a subsequence $k_l$ and a constant $c_2$
such that $I_1 (k_l) < c_2$. We then have that
\[
I_\nu^\theta (\nu_{k_l}) \leq \frac{1}{1-\theta} (16M + c_0 + c_2)
\]
for all $l$. Lemma~\ref{lem:frostman} now implies that $E_\alpha (x)
\in \mathscr{G}_\nu^\theta$.

From the statement that $E_\alpha (x) \in \mathscr{G}_\nu^\theta$
holds for a.e.\ $x$, we conclude in the same way as in the proof of
Theorem~\ref{the:randomlimsup} that $E_\alpha (x) \in
\mathscr{G}_\mu^\theta$ holds for a.e.\ $x$, with $\theta = (s
\alpha)^{-1}$. (We use that $\mu(C)$ can be made arbitrarily close to
$1$ and that $\beta$ and $t_0$ can be taken as close to $\alpha$ and
$s$ as we please.)

Finally, as in the proof of Theorem~\ref{the:randomlimsup}, by taking
$t_2$ close to $\udimh \mu$, we find that
\[
\dimh \bigcap_k E_\alpha (x_k) \geq \frac{1}{\alpha} \frac{\udimh
  \mu}{s}
\]
whenever the points $x_1, x_2, \ldots$ are elements of a set $A$ of
full $\mu$ measure.

\end{document}